\documentclass{article}

\pdfoutput = 1 

    \usepackage[final, nonatbib]{neurips_2020}

\usepackage[utf8]{inputenc} 
\usepackage[T1]{fontenc}    
\usepackage{hyperref}       
\usepackage{url}            
\usepackage{booktabs}       
\usepackage{amsfonts}       
\usepackage{nicefrac}       
\usepackage{microtype}      

\usepackage{mathtools}
\usepackage{amssymb}
\usepackage{amsthm}
\usepackage{mathabx}
\usepackage{float}
\usepackage{subfigure}
\usepackage{babel}
\usepackage{algorithm, algorithmic}  
\usepackage{color}
\usepackage[final]{changes}

\theoremstyle{plain}
\newtheorem{thm}{\protect\theoremname}[section]
\theoremstyle{definition}

\theoremstyle{plain}
\newtheorem{lem}[thm]{\protect\lemmaname}
\theoremstyle{plain}
\newtheorem{corollary}[thm]{\protect\corollaryname}
\theoremstyle{plain}
\newtheorem{assumption}[thm]{\protect\assumptionname}
\theoremstyle{plain}
\newtheorem{prop}[thm]{\protect\propositionname}
\theoremstyle{definition}

\theoremstyle{remark}

\providecommand{\assumptionname}{Assumption}
\providecommand{\definitionname}{Definition}
\providecommand{\examplename}{Example}
\providecommand{\lemmaname}{Lemma}
\providecommand{\propositionname}{Proposition}
\providecommand{\remarkname}{Remark}
\providecommand{\theoremname}{Theorem}
\providecommand{\corollaryname}{Corollary}

\title{A Feasible Level Proximal Point Method for Nonconvex Sparse Constrained Optimization}

\author{%
  Digvijay Boob\thanks{Work done when author was at Georgia Tech.} \\
  Southern Methodist University\\
  Dallas, TX \\
  \texttt{dboob@smu.edu} \\
   \And
   Qi Deng \\
   Shanghai university of Finance \& Economics\\
   Shanghai, China\\
   \texttt{qideng@sufe.edu.cn} \\
  \AND
  Guanghui Lan\\
  Georgia Tech\\
  Atlanta, GA\\
  \texttt{george.lan@isye.gatech.edu} \\
   \And
  Yilin Wang\\
  Shanghai university of Finance \& Economics\\
   Shanghai, China\\
 \texttt{2017110765@live.sufe.edu.cn} \\
}

\begin{document}

\DeclarePairedDelimiterX{\gnorm}[3]{\lVert}{\rVert_{#2}^{#3}}{#1}
\DeclarePairedDelimiter\abs{\lvert}{\rvert}
\DeclarePairedDelimiter{\bracket}{ [ }{ ] }
\DeclarePairedDelimiter{\paran}{(}{)}
\DeclarePairedDelimiter{\braces}{\lbrace}{\rbrace}
\DeclarePairedDelimiter{\floor}{\lfloor}{\rfloor}
\DeclarePairedDelimiter{\ceil}{\lceil}{\rceil}
\DeclarePairedDelimiterX{\inprod}[2]{\langle}{\rangle}{#1, #2}
\DeclarePairedDelimiterX{\inprodo}[3]{\langle}{\rangle_{#3}}{#1, #2}
\providecommand{\dx}[2]{\frac{d #1}{d #2}}
\providecommand{\dxx}[2]{\frac{d^2 #1}{d {#2}^2}}
\providecommand{\dxy}[3]{\frac{d^2 #1}{d {#2} d{#3}}}
\providecommand{\dox}[2]{\frac{\partial #1}{\partial #2}}
\providecommand{\doxx}[2]{\frac{\partial^2 #1}{\partial {#2}^2}}
\providecommand{\doxy}[3]{\frac{\partial^2 #1}{\partial {#2} \partial{#3}}}
\providecommand{\at}[3]{\left.#1\right\vert_{#2}^{#3}}
\providecommand{\br}[1]{{\left(#1\right)}}
\providecommand{\sbr}[1]{{\ast\left(#1\right)}}

\newcommand{\wb}[1]{{\widebar{#1}}}
\newcommand{\wt}[1]{{\widetilde{#1}}}
\newcommand{\grad}{\nabla}
\newcommand{\lyxdot}{.}

\global\long\def\inprod#1#2{\left\langle #1,#2\right\rangle }%

\global\long\def\norm#1{\left\Vert #1\right\Vert }%

\global\long\def\matr#1{\bm{#1}}%

\global\long\def\til#1{\tilde{#1}}%

\global\long\def\wtil#1{\widetilde{#1}}%

\global\long\def\wh#1{\widehat{#1}}%

\global\long\def\mcal#1{\mathcal{#1}}%

\global\long\def\mbb#1{\mathbb{#1}}%

\global\long\def\mtt#1{\mathtt{#1}}%

\global\long\def\ttt#1{\texttt{#1}}%

\global\long\def\br#1{\left(#1\right)}%

\global\long\def\norm#1{\left\Vert #1\right\Vert }%

\global\long\def\onebf{\mathbf{1}}%

\global\long\def\zerobf{\mathbf{0}}%



\global\long\def\rank{\text{rank}}%

\global\long\def\diam{\text{diam}}%

\global\long\def\epi{\text{epi }}%

\global\long\def\inte{\operatornamewithlimits{int}}%

\global\long\def\cov{\text{Cov}}%

\global\long\def\argmin{\operatornamewithlimits{argmin}}%

\global\long\def\argmax{\operatornamewithlimits{argmax}}%

\global\long\def\tr{\operatornamewithlimits{tr}}%

\global\long\def\dis{\operatornamewithlimits{dist}}%

\global\long\def\sign{\operatornamewithlimits{sign}}%

\global\long\def\prob{\text{Prob}}%

\global\long\def\st{\operatornamewithlimits{s.t.}}%

\global\long\def\dom{\text{dom}}%

\global\long\def\diag{\text{diag}}%

\global\long\def\and{\text{and}}%

\global\long\def\st{\text{s.t.}}%

\global\long\def\Var{\operatornamewithlimits{Var}}%

\global\long\def\raw{\rightarrow}%

\global\long\def\law{\leftarrow}%

\global\long\def\Raw{\Rightarrow}%

\global\long\def\Law{\Leftarrow}%

\global\long\def\vep{\varepsilon}%

\global\long\def\dom{\operatornamewithlimits{dom}}%

\global\long\def\tsum{{\textstyle {\sum}}}%

\global\long\def\Ebb{\mathbb{E}}%

\global\long\def\Nbb{\mathbb{N}}%

\global\long\def\Rbb{\mathbb{R}}%

\global\long\def\extR{\widebar{\mathbb{R}}}%

\global\long\def\Pbb{\mathbb{P}}%

\global\long\def\Acal{\mathcal{A}}%

\global\long\def\Bcal{\mathcal{B}}%

\global\long\def\Ccal{\mathcal{C}}%

\global\long\def\Dcal{\mathcal{D}}%

\global\long\def\Fcal{\mathcal{F}}%

\global\long\def\Gcal{\mathcal{G}}%

\global\long\def\Hcal{\mathcal{H}}%

\global\long\def\Kcal{\mathcal{K}}%

\global\long\def\Lcal{\mathcal{L}}%

\global\long\def\Mcal{\mathcal{M}}%

\global\long\def\Ncal{\mathcal{N}}%

\global\long\def\Ocal{\mathcal{O}}%

\global\long\def\Pcal{\mathcal{P}}%

\global\long\def\Scal{\mathcal{S}}%

\global\long\def\Tcal{\mathcal{T}}%

\global\long\def\Xcal{\mathcal{X}}%

\global\long\def\Ycal{\mathcal{Y}}%

\global\long\def\Ubf{\mathbf{U}}%

\global\long\def\Pbf{\mathbf{P}}%

\global\long\def\Ibf{\mathbf{I}}%

\global\long\def\Ebf{\mathbf{E}}%

\global\long\def\Abs{\boldsymbol{A}}%

\global\long\def\Qbs{\boldsymbol{Q}}%

\global\long\def\Lbs{\boldsymbol{L}}%

\global\long\def\Pbs{\boldsymbol{P}}%

\global\long\def\i{i}%

\maketitle

\begin{abstract}
Nonconvex sparse models have received significant attention in high-dimensional machine learning. 
In this paper, we study a new model consisting of a general convex or nonconvex objectives and a variety of continuous nonconvex sparsity-inducing constraints. For this constrained model, we propose a novel proximal point algorithm that solves a sequence of convex subproblems with gradually relaxed constraint levels. Each subproblem, having a proximal point objective and a convex surrogate constraint, can be efficiently solved based on a fast routine for projection onto the surrogate constraint. We establish the asymptotic convergence of the proposed algorithm to the Karush-Kuhn-Tucker (KKT) solutions. We also establish new convergence complexities to achieve an approximate KKT solution when the objective can be smooth/nonsmooth, deterministic/stochastic and convex/nonconvex with complexity that is on a par with gradient descent for unconstrained optimization problems in respective cases. To the best of our knowledge, this is the first study of the first-order methods with complexity guarantee for nonconvex sparse-constrained problems. We perform numerical experiments to demonstrate the effectiveness of our new model and efficiency of the proposed algorithm for large scale problems.\end{abstract}

\section{Introduction}
Recent years have witnessed a great deal of work on the sparse optimization arising from
machine learning, statistics and signal processing. 
A fundamental challenge in this area lies in finding the best set of size $k$ out of a total of $d$ ($k<d$) features to form a parsimonious fit to the data:
\begin{equation}\label{l0prob}
	\min\ \psi(x),\quad\text{subject to}\quad \gnorm{x}{0}{}\le k, x\in\Rbb^d.
\end{equation}
However, due to the discontinuity of $\gnorm{\cdot}{0}{}$ norm\footnote{Note that $\norm{\cdot}_0$ is not a norm in mathematical sense. Indeed, $\norm{x}_0 = \norm{tx}_0$ for any nonzero $t$.}, the above problem
 is intractable when there is no other assumptions. To bypass this difficulty, a popular approach is to replace
the $\ell_{0}$-norm by the $\ell_1$-norm,  giving rise to
an $\ell_{1}$-constrained or $\ell_{1}$-regularized 
problem.
A notable example is the Lasso (\cite{RN247}) approach for linear regression and its regularized variant
\begin{align} 
	&\min\ \gnorm{b-Ax}{2}{2}, \hspace{2em}\text{subject to} \ \gnorm{x}{1}{}\le \tau, x \in \Rbb^d; \label{lasso}\\
	&\min\ \gnorm{b-Ax}{2}{2} +\lambda \gnorm{x}{1}{}.\label{l1-regularization}
\end{align}
Due to the Lagrange duality theory, problem \eqref{lasso} and \eqref{l1-regularization} are equivalent in the sense that there is a one-to-one mapping between the parameters $\tau$ and $\lambda$.
A substantial amount of literature already exists for understanding the statistical properties of $\ell_1$ models  (\cite{zhao2006model,wainwright2009sharp,candes2009near,zhang2008sparsity,Kyrillidis12}) as well as for the development efficient algorithms when such models are employed (\cite{efron2004least,RN16,RN179,wright2009sparse, Kyrillidis12}).

In spite of their success, $\ell_{1}$ models suffer from the issue of biased estimation of large coefficients \cite{fan2001variable} and empirical merits of using nonconvex approximations were shown in \cite{scutari2017parallel}. 
Due to these observations, a large body of recent research looked at replacing the $\ell_1$-penalty in \eqref{l1-regularization} by a nonconvex function $g(x)$ to obtain sharper approximation of the $\ell_{0}$-norm:
\begin{equation} \label{noncvx-regularization}
	\min\ \psi(x) +\beta g(x),
\end{equation}
where , throughout the paper, $g(x)$ is a nonsmooth nonconvex function  of the form \[g(x) = \lambda \norm{x}_1 - h(x).\]
Here $h(x)$ is a convex and continuously differentiable function, giving $g(x)$ a DC form. This class of constraints already covers many important nonconvex sparsity inducing functions in the literature (see Table \ref{tab:constraint_fun_examples}).

Despite the favorable statistical properties (\cite{fan2001variable,zhang2010nearly,candes2007enhancing,zhang2012a}), nonconvex models have posed a great challenge for optimization algorithms and has been increasingly an important issue (\cite{gotoh2018dc,gong2013a,RN343,thi2015dc}).
While most of these works studied the regularized version, it is often favorable to consider the following constrained form:
\begin{equation} \label{noncvx-constraint}
	\min\ \psi(x), \quad\text{subject to} \quad g(x)\le \eta, x \in \Rbb^d,
\end{equation}
since sparsity of solutions is imperative in many applications of statistical learning and constrained form in \eqref{noncvx-constraint} explicitly imposes such a requirement. In contrast, \eqref{noncvx-regularization} imposes sparsity implicitly using penalty parameter $\beta$. However, unlike the convex problems, large values of $\beta$ do not necessarily imply small value of the nonconvex penalty $g(x)$.

Therefore, it is natural to ask \emph{whether we can provide an efficient algorithm for problem \eqref{noncvx-constraint}.} The continuous nonconvex relaxation \eqref{noncvx-constraint} of the $\ell_{0}$-norm in \eqref{l0prob}, albeit a straightforward one, was not studied in the literature. We suspect that to be the case due to the difficulty in handling nonconvex constraints algorithmically. There are two theoretical challenges: 
First, since the regularized form \eqref{noncvx-regularization} and the constrained form \eqref{noncvx-constraint} are not equivalent due to the nonconvexity of $g(x)$, we cannot bypass \eqref{noncvx-constraint} by solving problem \eqref{noncvx-regularization} instead.
Second, the nonconvex function $g(x)$ can be nonsmooth especially for the sparsity applications, presenting a substantial challenge for classic nonlinear programming methods, e.g., augmented Lagrangian methods and penalty methods (see \cite{RN27}) which assumes that functions are continuously differentiable.

{\bf Our contributions}\hspace{1em} In this paper, we study the newly proposed nonconvex constrained model \eqref{noncvx-constraint}. In particular, we present a novel level-constrained proximal point (LCPP) method for problem \eqref{noncvx-constraint} where the objective $\psi$ can be either deterministic/stochastic, smooth/nonsmooth and convex/nonconvex and the constraint $\{g(x) \le \eta\}$ models a variety of sparsity inducing nonconvex constraints proposed in the literature. The key idea is to translate problem \eqref{noncvx-constraint} into a sequence of convex subproblems where $\psi(x)$ is {\em convexified} using a proximal point quadratic term and $g(x)$ is {\em majorized} by a convex function $\wt{g}(x) [\ge g(x)]$. Note that $\{\wt{g}(x) \le \eta\}$ is a convex subset of the nonconvex set $\{g(x) \le \eta\}$.

We show that starting from a strict feasible point\footnote{Origin is always strictly feasible for sparsity inducing constraints and can be chosen as a starting point.}, LCPP traces a feasible solution path with respect to the set $\{g(x) \le \eta\}$. We also show that LCPP generates convex subproblems for which bounds on the optimal Lagrange multiplier (or the optimal dual) can be provided under a mild and a well-known constraint qualification. 
This bound on the dual and the proximal point update in the objective allows us to prove asymptotic convergence to the KKT points of the problem \eqref{noncvx-constraint}.

While deriving the complexity, we consider the inexact LCPP method that solves convex subproblems approximately. We show that the constraint, $\wt{g}(x) \le \eta$, has an efficient projection algorithm. Hence, each convex subproblem can be solved by projection-based first-order methods. 
This allows us to be feasible even when the solution reaches arbitrarily close to the boundary of the set $\{g(x)\le \eta\}$ which entails that the bound on the dual mentioned earlier works in the inexact case too. Moreover, efficient projection-based first-order method for solving the subproblem helps us get an accelerated convergence complexity of $O(\tfrac{1}{\vep}) [O(\tfrac{1}{\vep^2})]$ gradient [stochastic gradient] in order to obtain an $\vep$-KKT point. In particular, refer to Table \ref{tab:tab_conv}. We see that in the case where objective is smooth and deterministic, we obtain convergence rate of $O(1/\vep)$ whereas for nonsmooth and/or stochastic objective we obtain convergence rate of $O(1/\vep^2)$. This complexity is nearly the same as that of the gradient [stochastic gradient] descent for the regularized problem \eqref{noncvx-regularization} of the respective type. Remarkably, this convergence rate is better than black-box nonconvex function constrained optimization methods proposed in the literature recently (\cite{boob2019proximal,lin2019inexact}). See related work section for more detailed discussion. Note that the convergence of gradient descent does not ensure a bound on the infeasibility of the constraint $g$, whereas the KKT criterion requires feasibility on top of stationarity. Moreover, such a bound cannot be ensured theoretically due to the absence of duality. Hence, our algorithm provides additional guarantees without paying much in the complexity. 
\begin{table}[t]
	\centering
	\caption{Iteration complexities of LCPP for problem \eqref{noncvx-constraint} when the objective can be either convex or nonconvex, smooth or nonsmooth and  deterministic or stochastic}
	\label{tab:tab_conv}
	\begin{tabular}{c|c c|c c}
		\hline
		&\multicolumn{2}{|c|}{Convex \eqref{noncvx-constraint}}
		 &\multicolumn{2}{|c}{Nonconvex \eqref{noncvx-constraint}}\\\hline
		 Cases &Smooth &Nonsmooth &Smooth &Nonsmooth\\\hline
		 Deterministic &$O(1/\varepsilon)$ &$O(1/\varepsilon^2)$ &$O(1/\varepsilon)$ &$O(1/\varepsilon^2)$\\\hline
		 Stochastic &$O(1/\varepsilon^2)$ &$O(1/\varepsilon^2)$ &$O(1/\varepsilon^2)$ &$O(1/\varepsilon^2)$\\\hline
	\end{tabular}
\end{table}

We perform numerical experiments to measure the efficiency of our LCPP method and the effectiveness of the new constrained model \eqref{noncvx-constraint}. First, we show that our algorithm has competitive running time performance against open-source solvers, e.g., DCCP \cite{shen2016disciplined}. Second, we also compare the effectiveness of our constrained model with respect to the existing convex and nonconvex regularization models in the literature. Our numerical experiments show promising results compared to $\ell_1$-regularization model \ref{l1-regularization} and has competitive performance with respect to recently developed algorithm for nonconvex regularization model \ref{noncvx-regularization} (see \cite{gong2013a}). Given that this is the first study in the development of algorithms for the constrained model, we believe empirical study of even more efficient algorithms solving problem \eqref{noncvx-constraint} may be of independent interest and can be pursued in the future.
%
%

{\bf Related work} \hspace{1em}
There is a growing interest in using convex majorization for solving nonconvex optimization with nonconvex function constraints.
Typical frameworks include difference-of-convex (DC) programming (\cite{thi2018dc}), majorization-minimization (\cite{sun2017majorization}) to name a few.  Considering the substantial literature, we  emphasize the most relevant work to our current paper.
Scutari et al. \cite{scutari2017parallel} proposed general approaches to majorize nonconvex constrained problems and include \eqref{noncvx-constraint} as a special case. They require exact solutions of the subproblems and prove asymptotic convergence which is prohibitive for large-scale optimization. 
Shen et al. \cite{shen2016disciplined} proposed a disciplined convex-concave programming (DCCP) framework for a class of DC programs in which \eqref{noncvx-constraint} is a special case. 
Their work is empirical and does not provide specific convergence results.

The more recent works \cite{boob2019proximal,lin2019inexact} considered a more general problems where $g(x) = \wt{h}(x) - h(x)$ for some general convex function $\wt{h}$. They propose a type of proximal point method in which large enough quadratic proximal term is added into both objective and constraint in order to obtain a convex subproblem. This convex function constrained subproblem can be solved by oracles whose output solution might have small infeasibility. Moreover these oracles have weaker convergence rates due to generality of function $\wt{h}$ over $\ell_{1}$. Complexity results proposed in these works, when applied to problem \eqref{noncvx-constraint}, entail $O(1/\vep^{3/2})$ iterations for obtaining an $\vep$-KKT point under a {\em strong feasibility} constraint qualification.
In similar setting, we show faster convergence result of $O(1/\vep)$. This due to the fact that our oracle for solving the subproblem is more efficient than those used in their paper. We can obtain such an oracle due to two reasons: i) convex surrogate constraint $\wt{g}$ in LCPP majorizes the constraint differently than adding the proximal quadratic term, ii) presence of $\ell_1$ in the form of $g(x)$ allows for developing an efficient projection mechanism onto the chosen form of $\wt{g}$. Moreover, our convergence results hold under a well-known constraint qualification which is weaker compared to {\em strong feasibility} since our oracle outputs a feasible solution whereas they can get a solution which is slightly infeasible. 

There is also a large body of work on directly optimizing the $\ell_0$ constraint problem
\cite{bertsimas2016best,blumensath2008iterative, foucart2011hard,yuan2017gradient,zhou2018efficient}. While \cite{bertsimas2016best} can be quite good for small dimension $d = 1000$s, it remains unclear how to scale up for larger datasets. Other methods are part of the hard-thresholding algorithms, requiring additional assumptions such as {\em Restricted Isometry Property}. These research areas, though  interesting, are not related to the continuous optimization setting where large-scale problems can be solved relatively easily. Henceforth, we only focus on the continuous approximations of $\ell_0$-norm.

{\bf Structure of the paper}\hspace{1em}
Section \ref{sec:Setup} presents the problem setup and preliminaries. 
Section \ref{sec:Main-algo} introduces LCPP method and shows the asymptotic
convergence, convergence rates and the boundedness of the optimal dual. 
Section \ref{sec:Experiments} presents numerical results.
Finally, Section \ref{sec:Conclusion} draws conclusion.

\section{Problem setup\label{sec:Setup}}
Our main goal is to solve problem \eqref{noncvx-constraint}. We make Assumption \ref{ass:general} throughout the paper.
\begin{assumption}\label{ass:general}
1. $\psi(x)$ is a continuous and possibly nonsmooth nonconvex function 
satisfying
		: 
		\begin{equation}\label{eq:low_curv_psi}
		\psi(x)\ge\psi(y)+\inprod{\psi'(y)}{x-y} -\tfrac{\mu}{2}\norm{x-y}^2.
		\end{equation}  
		2. $g(x)$ is a nonsmooth nonconvex function of the form $g(x)=\lambda\|x\|_{1}-h(x)$, where $h(x)$ is convex and continuously differentiable. 
\end{assumption}
\begin{table}[h]
	\caption{Examples of constraint function $g(x) = \lambda\gnorm{x}{1}{} -h(x)$. 
	}
	\label{tab:constraint_fun_examples}
	\centering
	\begin{tabular}{c c l}\hline
		Function $g(x)$ & Parameter $\lambda$ & Function $h(x)$\\\hline
		MCP\cite{zhang2010nearly}& $\lambda$ & $h_{\lambda,\theta}(x) = \begin{cases}
		\frac{x^{2}}{2\theta} & \text{if }|x|\le\theta\lambda,\\
		\lambda\left|x\right|-\frac{\theta\lambda^{2}}{2} & \text{if }\abs{x} > \theta\lambda.
		\end{cases}\label{eq:mcp-concave}$\\
		SCAD\cite{fan2001variable} &$\lambda$ &$h_{\lambda, \theta}(x) = \begin{cases} 0 & \text{if }|x|\le\lambda,\\ \tfrac{x^{2}-2\lambda|x|+\lambda^{2}}{2(\theta-1)} & \text{if }\lambda<|x|\le\theta\lambda,\\ \lambda|x|-\tfrac{1}{2}(\theta+1)\lambda^{2} & \text{if }|x|>\theta\lambda. \end{cases}\label{eq:scad}$\\[6mm]
		Exp\cite{Bradley98} &$\lambda$ &$h_\lambda(x) =  e^{-\lambda \abs{x}}-1 + \lambda\abs{x}$.\\[0.5mm]
		Log\cite{Weston02} &$\frac{\theta}{\log(1+\theta)}$ &$h_{\theta}(x) = \frac{\theta}{\log(1+\theta)}\abs{x} - \frac{\log(1+\theta\abs{x})}{\log(1+\theta)}$.\\[1.5mm]
		$\ell_p(0<p<1)$\cite{Fu98}&$\frac{\varepsilon^{1/\theta-1}}{\theta}$ &$h_{\varepsilon,\theta}(x) = \frac{\varepsilon^{1/\theta-1}}{\theta}\abs{x} - (\abs{x}+\varepsilon)^{1/\theta}$.\\[1mm]
		$\ell_p (p<0)$\cite{Rao99} &$-p\theta$ &$h_{\theta}(x) = -p\theta\abs{x} - 1+(1+\theta\abs{x})^p$.\\\hline
	\end{tabular}
\end{table}
\begin{figure}[t]
	\caption{Graphs for various constraints along with $\ell_1$. For $\ell_p (0<p<1)$, we have $\vep = 0.1$ .}
	\label{fig:constraint_plots}
	\centering
	\includegraphics[width=0.6\linewidth]{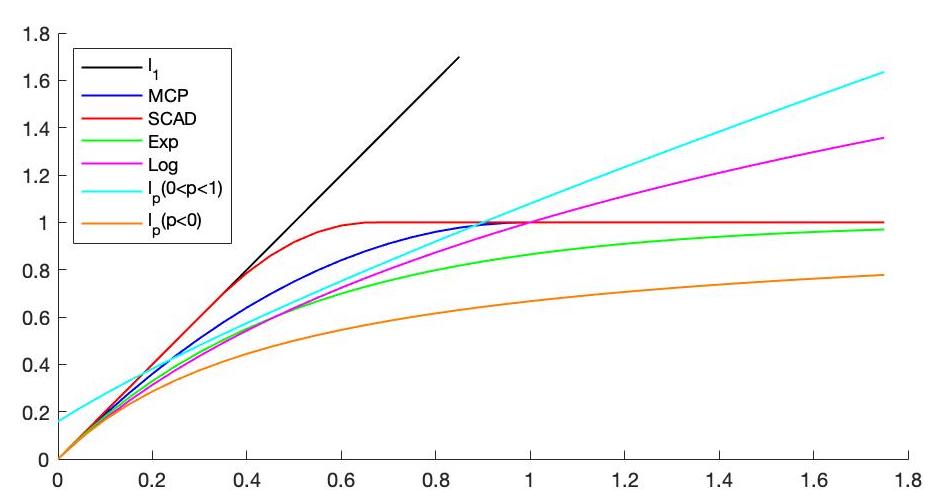}
\end{figure}

{\bf Notations}\hspace{1em}
We use $\gnorm{\cdot}{}{}$ to denote standard Euclidean norm whereas $\ell_{1}$-norm is denoted as $\gnorm{\cdot}{1}{}$.
The Lagrangian function for problem \eqref{noncvx-constraint} is defined as $\Lcal(x,y)=\psi(x)+y(g(x)-\eta)$ where $y \ge 0$.
For nonconvex nonsmooth function $g(x)$ in the form of (\ref{eq:mcp-concave}), we denote its {\em subdifferential}\footnote{Various subdifferentials exist in the literature for nonconvex optimization problem. Here, we use subdifferential Definition 3.1 in Boob et al. \cite{boob2019proximal} for nonconvex nonsmooth function $g$.} by 
$\partial g(x)=\partial\paran{\lambda\gnorm {x}{1}{}}-\nabla h(x)$. For this definition of subdifferential, we consider the following KKT condition:

{\bf The KKT condition} \hspace{1em}For Problem \eqref{noncvx-constraint}, we say that $x$ is the
(stochastic) $(\vep, \delta)$- KKT solution if there exists $\bar{x}$ and $\bar{y}\ge 0$ such that $g(\bar{x}) \le\eta$, $\Ebb\,\|x-\bar{x}\|^{2}  \le\delta$
\begin{equation}
\begin{aligned}\Ebb\left|\bar{y}\left[g(\bar{x})-\eta\right]\right| & \le\vep\\
\Ebb\left[\dis\left(\partial_{x}\Lcal(\bar{x},\bar{y}),0\right)\right]^{2} & \le\vep
\end{aligned}
\label{eq:stoc-approx-kkt}
\end{equation}
Moreover, for $\vep = \delta = 0$, we have that $\bar{x}$ is the KKT solution or satisfied KKT condition. If $\delta = O(\vep)$, we refer to this solution as an $\vep$-KKT solution in order to be brief.\\
It should be mentioned that local or global optimality does not generally imply the
KKT condition. However, KKT condition is shown to be necessary for optimality when Mangasarian-Fromovitz constraint qualification (MFCQ) holds \cite{boob2019proximal}. 
Below, we make MFCQ assumption precise:
\begin{assumption}[MFCQ \cite{boob2019proximal}]\label{ass:mfcq}
Whenever
the constraint is active: $g(\bar{x})=\eta$, there exists a direction
$z$ 
such that
$
\max_{v\in\partial g(\bar{x})}v^{T}z<0.
$
\end{assumption}\vspace{-.5em}
For differentiable $g$, MFCQ requires existence of $z$ such that $z^T\grad g(\wb{x}) < 0$, reducing to the classical form of MFCQ \cite{RN27}. Below, we summarize necessary optimality condition under MFCQ.
\begin{prop}[Necessary condition \cite{boob2019proximal}]
	\label{thm:MFCQ}
	Let $\bar{x}$ be a local optimal solution of problem \eqref{noncvx-constraint}. If $\bar{x}$ satisfies Assumption \ref{ass:mfcq},
	then there exists $\bar{y} \ge 0$ such that 
	\eqref{eq:stoc-approx-kkt} holds with $\vep = \delta = 0$.
\end{prop}

\section{A novel proximal point algorithm\label{sec:Main-algo}}

\begin{algorithm}[h]
	\caption{Level constrained proximal point {\bf (LCPP)} method}
	\label{alg:main}
	\begin{algorithmic}[1]
		\STATE {\bfseries Input:} $x^{0}=\hat{x}$, $\gamma>0$, $\eta_{0}<\eta$
		\FOR{$k=1$ {\bfseries to} $K$}
		\STATE Set $\eta_k = \eta_{k-1} + \delta_{k}$;
		\STATE  $g_k(x) :=\lambda\norm x_{1}-h(x^{k-1})-\nabla h(x^{k-1})^T(x-x^{k-1})$;
		\STATE Return feasible solution $x^k$ of the problem \vspace{-0.5em}
		\begin{equation}\label{subprob}
		\min \psi_{k}(x)=\psi(x)+\tfrac{\gamma}{2}\|x-x^{k-1}\|^{2},\hspace{2em}
		\text{subject to} \quad g_k(x) \le \eta_{k}
		\end{equation}\vspace{-1.5em}
		\ENDFOR
	\end{algorithmic}
\end{algorithm}

%
%
%
%
%
%
%
LCPP method solves a sequence of convex subproblems \eqref{subprob}. 
In particular, note that $g_{k}(x)$ 
majorizes $g(x)$: 
$
g_{k}(x)  \ge g(x),\
g_{k}(x^{k-1})  = g(x^{k-1}).
$
implying that $\{g_k(x) \le \eta_k\}$ is a convex subset of the original problem. 
It can also be observed that adding a proximal term in the objective yields $\psi_{k}$ strongly convex for large enough $\gamma>0$. In the current form, Algorithm \ref{alg:main} requires a feasible solution of \eqref{subprob} and requirement of sequence $\{\eta_k\}$ is left unspecified. 
We first
make the following assumptions.
\begin{assumption}
	[Strict feasibility]\label{assu:feasibility-level} There exist sequence $\{\eta_{k}\}_{k \ge 0}$ satisfying:\\ 
	1. $\eta_{0}<\eta$ and a point $\hat{x}$ of such that
	$g(\hat{x})<\eta_{0}.$\\
	2. The sequence $\{\eta_{k}\}$ is monotonically increasing and converges
	to $\eta$: $\lim_{k\raw\infty}\eta_{k}=\eta$.
\end{assumption}
In light of Assumption \ref{assu:feasibility-level}, starting from
a strictly feasible point $x^{0}$, Algorithm \ref{alg:main} solves
subproblems \eqref{subprob} with gradually relaxed constraint levels.
This allows us to assert that each subproblem is strictly feasible\footnote{For specific examples of $g$, we show that origin is always the most feasible (and strictly feasible) solution of each subproblem and hence, does not require the predefined level-routine of LCPP to assert strict feasibility of subproblem. However, in order to keep generality of discussion, we perform the analysis under the level-setting.}. Indeed, we have $g_k(x^k) \le \eta_k \Rightarrow g_{k+1}(x^k) = g(x^k) \le g_k(x^k) \le \eta_k < \eta_{k+1}$. This implies the existence of KKT solution for each subproblem. A formal statement can be found in the appendix. Moreover, all the proofs of our technical results can be found in the appendix and we just make statements in the main article henceforth. 

{\bf Asymptotic convergence of LCPP method and boundedness of the optimal dual}

Our next goal is to establish asymptotic convergence of Algorithm
\ref{alg:main} to the KKT points. To this end, we require a
uniform boundedness assumption on the Lagrange multipliers. First, we prove asymptotic convergence under this assumption then we justify it under MFCQ. 
Before stating the convergence results, we make the following boundedness assumption.

\begin{assumption}
[Boundedness of dual variables]\label{assu:bound-y-1}There exists
$B>0$ such that $\sup_{k}\bar{y}^{k}<B$ a.s.
\end{assumption}
For the deterministic case, we remove the measurablity part in the above assumption and assert that $\sup_{k} \bar{y}^k < B$. The following asymptotic convergence theorem is in order.
\begin{thm}[Convergence to KKT]\label{thm:conv_asym_KKT}
 Let $\pi_k$ denotes the randomness of $x^1,x^2,...,x^{k-1}$. Assume that there exists a  $\rho \in[0, \gamma-\mu]$  and a summable nonnegative sequence $\zeta_k$  
such that
\begin{equation} \label{inexactness}
\Ebb\bracket{\psi_{k}(x^{k})-\psi_{k}(\bar{x}^{k})|\pi_k}\le \tfrac{\rho}{2}\gnorm{\bar{x}^k-x^{k-1}}{}{2} + \zeta_{k}.
\end{equation}
Then, under Assumption \ref{assu:feasibility-level} and \ref{assu:bound-y-1} 
for any limit point $\wtil x$ of the proposed algorithm, there exists
a dual variable $\wtil y$ such that $(\wtil x,\wtil y)$ satisfies
KKT condition, almost surely. 
\end{thm}\vspace{-0.5em}
This theorem shows that any limit point of Algorithm \ref{alg:main} converges to a KKT point. However, it makes the 
assumption that dual is bounded. Since the optimal dual depends on the convex subproblems \eqref{subprob} which are generated dynamically in the algorithm, 
it is important to justify Assumption \ref{assu:bound-y-1}. 
To this end, we show that Assumption \ref{assu:bound-y-1} is satisfied under a well-known constraint qualification.\deleted{ when Algorithm \ref{alg:main} is deterministic, i.e., \eqref{inexactness} is satisfied deterministically.}
\begin{thm}[Boundedness condition]
	\label{thm:bound_dual_MFCQ_new}
	Suppose Assumption \eqref{assu:feasibility-level} and relation \eqref{inexactness} are satisfied and all limit points of Algorithm \ref{alg:main} \added{exists a.s.,} and satisfy the MFCQ condition. Then, $\wb{y}^k$ is bounded a.s.
\end{thm}\vspace{-0.5em}
This theorem shows the existence of dual under the MFCQ assumption for all limit points of Algorithm \ref{alg:main}. MFCQ is a mild constraint qualification frequently used in the existing literature \cite{RN27}. 
In certain cases, we also provide explicit bounds on the dual variables using the fact that origin is most feasible solution to the subproblem. These bounds quantify how ``closely" the MFCQ assumption is violated and provides explicitly the effect on the magnitude of the optimal dual. Additional results and discussion in this regard are deferred to the Appendix \ref{apx:bound_dual_overall}. For the purpose of this article, we assume that the dual variables remain bounded henceforth.

{\bf Complexity of LCPP method \label{sec:Complexity}}

Our goal here is to analyze the complexity of the proposed algorithm. Apart from the negative lower curvature guarantee \eqref{eq:low_curv_psi} of the objective function, we impose that $h$ has Lipschitz continuous gradients,
$\norm{\grad h(x) - \grad h(y)} \le L_h\norm{x-y}.$ This is satisfied by all functions in Table \ref{tab:constraint_fun_examples}. 
Now we discuss a general convergence result of LCPP method for original nonconvex problem \eqref{noncvx-constraint}. 
\begin{thm}
	\label{thm:complexity_main}
	Suppose Assumption \ref{assu:feasibility-level} and \ref{assu:bound-y-1}  hold such that $\delta_{k} = \tfrac{\eta-\eta_{0}}{k(k+1)}$ for all $k \ge 1$. Let $x^k$ satisfy \eqref{inexactness} where $\rho \in [0, \gamma-\mu)$ and $\{\zeta_{k}\}$ is a summable nonnegative sequence. Moreover, $x^k$ is a feasible solution of the $k$-th subproblem, i.e.,
	\begin{equation}
	g_k(x^k) \le \eta_{k}. \label{eq:subprob_feas}
	\end{equation}
	If $\hat{k}$ is chosen uniformly at random from $\left\lfloor \tfrac{K+1}{2} \right\rfloor$ to K then there exists a pair $(\wb{x}^{\hat{k}}, \wb{y}^{\wh{k}})$ satisfying
	\begin{align*}
	\Ebb\bracket{ \dis\paran{\partial_{x}\Lcal(\bar{x}^{\hat{k}},\bar{y}^{\hat{k}}),0}^{2} } &\le \tfrac{16(\gamma^2 + B^2L_h^2)}{K(\gamma-\mu-\rho)}\big(\tfrac{\gamma-\mu + \rho}{2(\gamma-\mu)}\Delta^0 + Z\big),\\
	\Ebb\bracket{\bar{y}^{\hat{k}}\abs{g(\bar{x}^{\hat{k}})-\eta } } &\le \tfrac{2BL_h}{K(\gamma-\mu-\rho)}\big(\tfrac{\gamma-\mu + \rho}{\gamma-\mu}\Delta^0 + 2Z\big) + \tfrac{2B(\eta-\eta_{0})}{K},\\
	\Ebb\bracket{\gnorm{x^{\hat{k}}-\bar{x}^{\hat{k}}}{}{2}}&\le \tfrac{4\rho(\gamma-\mu+\rho)}{K(\gamma-\mu)^2(\gamma-\mu-\rho)}\Delta^0 + \tfrac{8Z}{K(\gamma-\mu-\rho)},
	\end{align*}
	where, $\Delta^0 := \psi(x^0) - \psi(x^*)$, $Z := \tsum_{k=1}^K \zeta_{k}$ and expectation is taken over the randomness of $\wh{k}$ and solutions $x^k$, $k = 1 ,\dots, K$.
\end{thm}
Note that Theorem \ref{thm:complexity_main} assumes that subproblem \eqref{subprob} can be solved according to the framework of \eqref{inexactness} and \eqref{eq:subprob_feas}. When the subproblem solver is deterministic then we ignore the expectation in \eqref{inexactness}. It is easy to see from the above theorem that for $x^{\hat{k}}$ to be an $\vep$-KKT point, we must have $K = O(1/\vep)$ and $\zeta_{k}$ must be small enough such that $Z$ is bounded above by a constant. The complexity analysis of different cases now boils down to understanding the number of iterations of the subproblem solver needed in order to satisfy these requirements on $\rho$ and $\{\zeta_{k}\}$ (or $Z$).

In the rest of this section, we provide a unified complexity result for solving subproblem \eqref{subprob} in Algorithm \ref{alg:main} such that criteria in \eqref{inexactness} and \eqref{eq:subprob_feas} are satisfied for various settings of the objective $\psi(x)$.

{\bf Unified method for solving subproblem \eqref{subprob}}\hspace{1em}
Here we provide a unified complexity analysis for solving subproblem \eqref{subprob}. In particular, consider the form of the objective
$\psi(x)=\Ebb_\xi \bracket{\Psi(x,\xi)},$
 where $\xi$ is the random input of $\Psi(x,\xi)$ and $\psi(x)$ satisfies the following property:
 \[\psi(x) - \psi(y) - \inprod{\psi'(y)}{x-y} \le \tfrac{L}{2}\norm{x-y}^2 + M \norm{x-y}. \]
 Note that, when $M = 0$, function $\psi$ is Lipschitz smooth whereas when $L = 0$, it is nonsmooth. Due to the possible stochastic nature of $\Psi$, negative lower curvature in \eqref{eq:low_curv_psi} and the combined smoothness and nonsmoothness property above, we have that $\psi$ can be either smooth or nonsmooth, deterministic or stochastic and convex ($\mu = 0$) or nonconvex ($\mu > 0$). We also assume bounded second moment stochastic oracle for $\psi'$ when $\psi$ is a stochastic function: For any $x$, we have an oracle whose output, $\Psi'(x,\xi)$, satisfies $\Ebb_\xi[\Psi'(x,\xi)] = \psi'(x)$ and $\Ebb\bracket{\gnorm{\Psi'(x,\xi) -\psi'(x)}{}{2}} \le \sigma^2$.
 
For such a function, we consider an accelerated stochastic approximation algorithm (AC-SA) proposed in \cite{ghadimi12} for solving the subproblem \eqref{subprob} which can be reformulated as $\min_x \psi_k(x) +\Ibf_{\{g_k(x) \le \eta_k\}}(x),$
where $\Ibf$ is the indicator set function. AC-SA algorithm can be applied when $\gamma \ge \mu$. In particular, $\psi_k(x) := \psi(x) + \tfrac{\gamma}{2} \norm{x-x^{k-1}}^2$ is $(\gamma-\mu)$-strongly convex and $(L+\gamma)$-Lipschitz smooth.
Moreover, AC-SA requires computation of a single prox operation of the following form in each iteration:
\begin{equation} \label{eq:prox}
\argmin_x w^Tx + \norm{x-\wb{x}}^2+ \Ibf_{\{g_k(x) \le \eta_k\}}(x),
\end{equation}
for any $w, \wb{x} \in \Rbb^d$. We show an efficient method for solving this problem at the end of in this section. 
For now, we look at convergence properties of the AC-SA:
\begin{prop}\cite{ghadimi12}
	\label{prop:AC-SA}
	Let $x^k$ be the output of AC-SA algorithm after running $T_k$ iterations for the subproblem \eqref{subprob}. Then $g_k(x^k) \le \eta_k$ and $\Ebb[\psi_{k}(x^k) - \psi_k(\wb{x}^k)] \le \tfrac{2(L+\gamma)}{T_k^2}\norm{x^{k-1}-\wb{x}^k}^2 + \tfrac{8(M^2+\sigma^2)}{(\gamma-\mu)T_k}$
\end{prop}
Note that convergence result in Proposition \ref{prop:AC-SA} closely follows the requirement in \eqref{inexactness}.
In particular, we should ensure that $T_k$ is big enough such that $\tfrac{\rho}{2} \le \tfrac{2(L+\gamma)}{T_k^2}$ and $\zeta_{k} = \tfrac{8(M^2+\sigma^2)}{(\gamma-\mu)T_k}$ sum to a constant. Consequently, we have the following corollary:
\begin{corollary}\label{cor:unified_complexity_main}
	Let $\psi$ be nonconvex such that it satisfies \eqref{eq:low_curv_psi} with $\mu > 0$. Set $\gamma = 3\mu$ and run AC-SA for $T_k = \max\{ 2\paran[\big]{\tfrac{L}{\mu}+3}^{1/2}, K(M+\sigma)\}$ iterations where $K$ is total iterations of Algorithm \ref{alg:main}. Then, we obtain that $x^{\hat{k}}$ is an $(\vep_1, \vep_2)$-KKT point of \eqref{noncvx-constraint}, where $\hat{k}$ is chosen according to Theorem \ref{thm:complexity_main} and
	\begin{align*}
	\vep_1 &= \big( \tfrac{3\Delta^0}{2K} + \tfrac{8(M+\sigma)}{\mu K} \big) \max\big\{\tfrac{8(9\mu^2+ B^2L_h^2)}{\mu}, \tfrac{2BL_h}{\mu}\}+ \tfrac{2B(\eta-\eta_{0})}{K},\quad	\vep_2 = \tfrac{3\Delta^0}{\mu K} + \tfrac{32(M+\sigma)}{\mu^2 K}
	\end{align*}
\end{corollary}
 Note that Corollary \ref{cor:unified_complexity_main} gives a unified complexity for obtaining KKT point of \eqref{noncvx-constraint} in various settings of nonconvex objective $(\mu > 0)$. First, in order to get an $\vep$-KKT point, $K$ must be of $O(1/\vep)$. If the problem is deterministic and smooth then $M = \sigma = 0$. In this case, $T_k = 2\paran{\tfrac{L}{\mu} + 3}^{1/2}$ is a constant. Hence, the total iteration count is $\tsum_{k = 1}^K T_k = O(K)$, implying that total iteration complexity for obtaining an $\vep$-KKT point is of $O(1/\vep)$.
 For nonsmooth or stochastic cases, $M$ or $\sigma$ is positive. Hence, $T_k = O(K(M+\sigma))$ implying the total iteration complexity $\tsum_{k=1}^K T_k = O(K^2)$, which is of $O(1/\vep^2)$. Similar result for the convex case is shown in the appendix.

{\bf Efficient projection} \hspace{1em}We conclude this section by formally stating the theorem which provides an efficient oracle for solving the projection problem \eqref{eq:prox}. Since $g_k(x) = \lambda \gnorm{x}{1}{} + \inprod{v}{x}$, the linear form along with $\ell_1$ ball breaks the symmetry around origin which is used in existing results on (weighted) $\ell_1$-ball projection \cite{duchi2008efficient,kopsinis10onlinesparse}. Our method involves a careful analysis of Lagrangian duality equations to convert the problem into finding the root of a piecewise linear function. Then a line search method can be employed to find the solution in $O(d\log{d})$ time. 
The formal statement is as follows: 
\begin{thm}\label{thm:prox_oracle}
	There exists an algorithm that runs in $O(d\log{d})$-time and solves the following problem exactly:
	\begin{equation}\label{projection}
	\min_{x \in \Rbb^d }\, \tfrac{1}{2}\norm{x-v}^2 \ {\text{subject to}}\ \norm{x}_1+\inprod{u}{x}\le \tau.
	\end{equation}
\end{thm}
In conclusion, note that \eqref{eq:prox} and \eqref{projection} are equivalent where $v$ in \eqref{projection} can be replaced by $\wb{x} + \tfrac{1}{2}w$ of \eqref{eq:prox} to get the equivalence of the objective functions of the two problems. 

\section{Experiments\label{sec:Experiments}}
The goal of this section is to illustrate the empirical performance of LCPP. For simplicity, we will consider the following  learning problem: 
\[\min_x\, \psi(x)=\tfrac{1}{n}\tsum_{i = 1}^n L_i(x),\quad \text{s.t.}\quad g(x)\le \eta, \] where $L_i(x)$ denotes the loss function. Specifically, we consider logistic loss $L_i(x)=\log(1+\exp(-b_i a_i^Tx))$ for classification and squared loss $L_i(x)=(b_i-a_i^Tx)^2$ for regression. Here $(a_i, b_i)$ is the training sample, and $g(x)$ is the MCP penalty (see Table~\ref{tab:constraint_fun_examples}).  
Details of the testing datasets are summarized in Table \ref{tab:Datasets}. 
As we have stated, LCPP can be equipped with projected first order methods for fast iteration. We compare the efficiency of (spectral) gradient descent \cite{gong2013a}, Nesterov accelerated gradient and stochastic gradient \cite{xiao2014proximal} for solving LCPP subproblem. We find that spectral gradient outperforms the other methods in the logistic regression model and hence use it in LCPP for the remaining experiment for the sake of simplicity. Due to the space limit, we leave the discussion of this part in appendix.
The rest of the section will compare the optimization efficiency of LCPP with the state-of-the-art nonlinear programming solver, and compare the proposed sparse constrained models solved by LCPP with standard convex and nonconvex sparse regularized models. 
\begin{table}[h]
\caption{\label{tab:Datasets}Dataset description. R for regression and C for classification. \texttt{mnist} is formulated as a binary 
problem to classify digit $5$ from the other digits. \texttt{real-sim} is randomly partitioned into $70\%$ training data and $30\%$ testing data.}	
\centering{}%
\begin{tabular}{cccccc}
\toprule 
Datasets & Training size &Testing size &  Dimensionality & Nonzeros & Types \tabularnewline
\midrule
 \texttt{real-sim}  &  50347 & 21962 & 20958 & 0.25\%  & C \tabularnewline
 \midrule
 \texttt{rcv1.binary} &  20242 & 677399 &  47236 & 0.16\% & C \tabularnewline
 \midrule 
  \texttt{mnist} & 60000 & 10000 & 784  & 19.12\% & C \tabularnewline
  \midrule
  \texttt{gisette} & 6000 & 1000 & 5000 & 99.10\% & C  \tabularnewline
    \midrule
  \texttt{E2006-tfidf} & 16087 & 3308  & 150360 & 0.83\% & R  \tabularnewline
    \midrule
  \texttt{YearPredictionMSD} & 463,715 & 51,630  & 90 &  100\% & R  \tabularnewline
  \bottomrule
\end{tabular}
\end{table}\\
Our first experiment is to compare LCPP with existing optimization library for their optimization efficiency.   To the best of our knowledge, DCCP (\cite{shen2016disciplined})
is the only open-source package available for the proposed nonconvex constrained problem.
While the work~\cite{shen2016disciplined} has made its code available
online, we found that their
code had unresolved errors in parsing MCP functions. Therefore, we replicate their setup in our own implementation.
DCCP converts the initial problem into a sequence of relatively easier convex problems amenable to CVX (\cite{diamond2016cvxpy}), a convex optimization interface that runs on top of popular optimization libraries. We choose DCCP with MOSEK as the backend as it consistently outperforms DCCP with the default open-source solver SCS.


\begin{figure}[h]
	\includegraphics[scale=0.22]{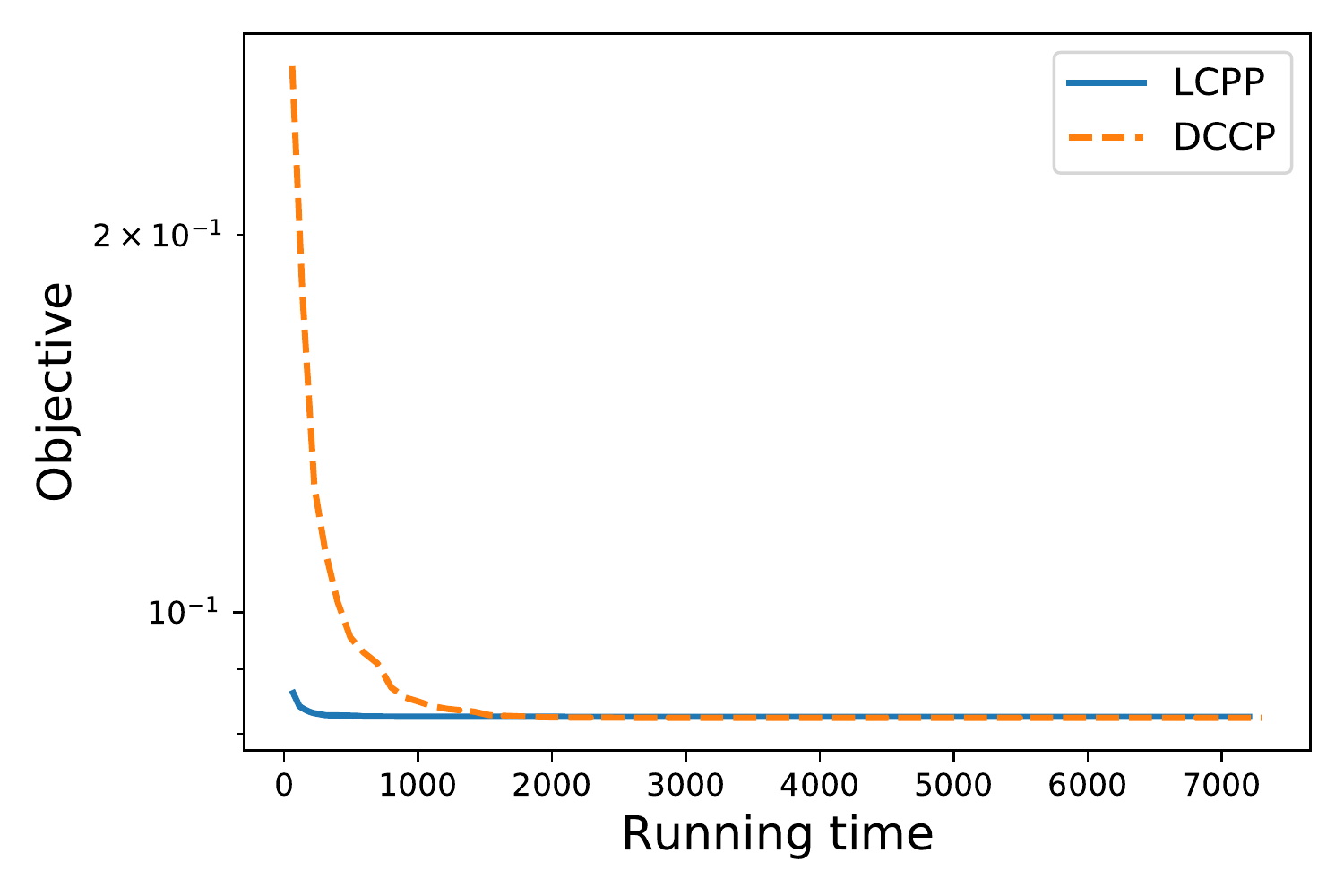}
	\includegraphics[scale=0.22]{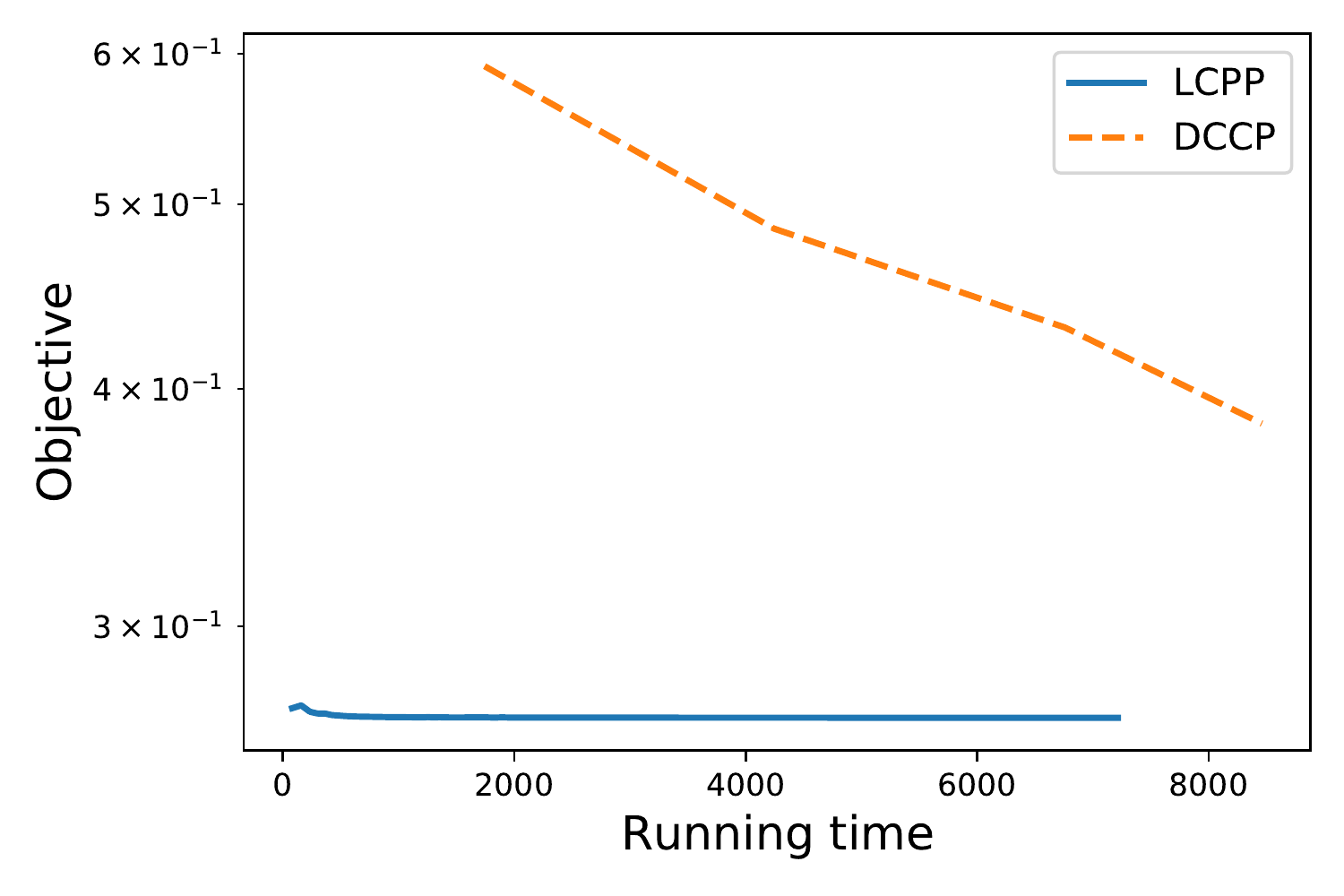}
	\includegraphics[scale=0.22]{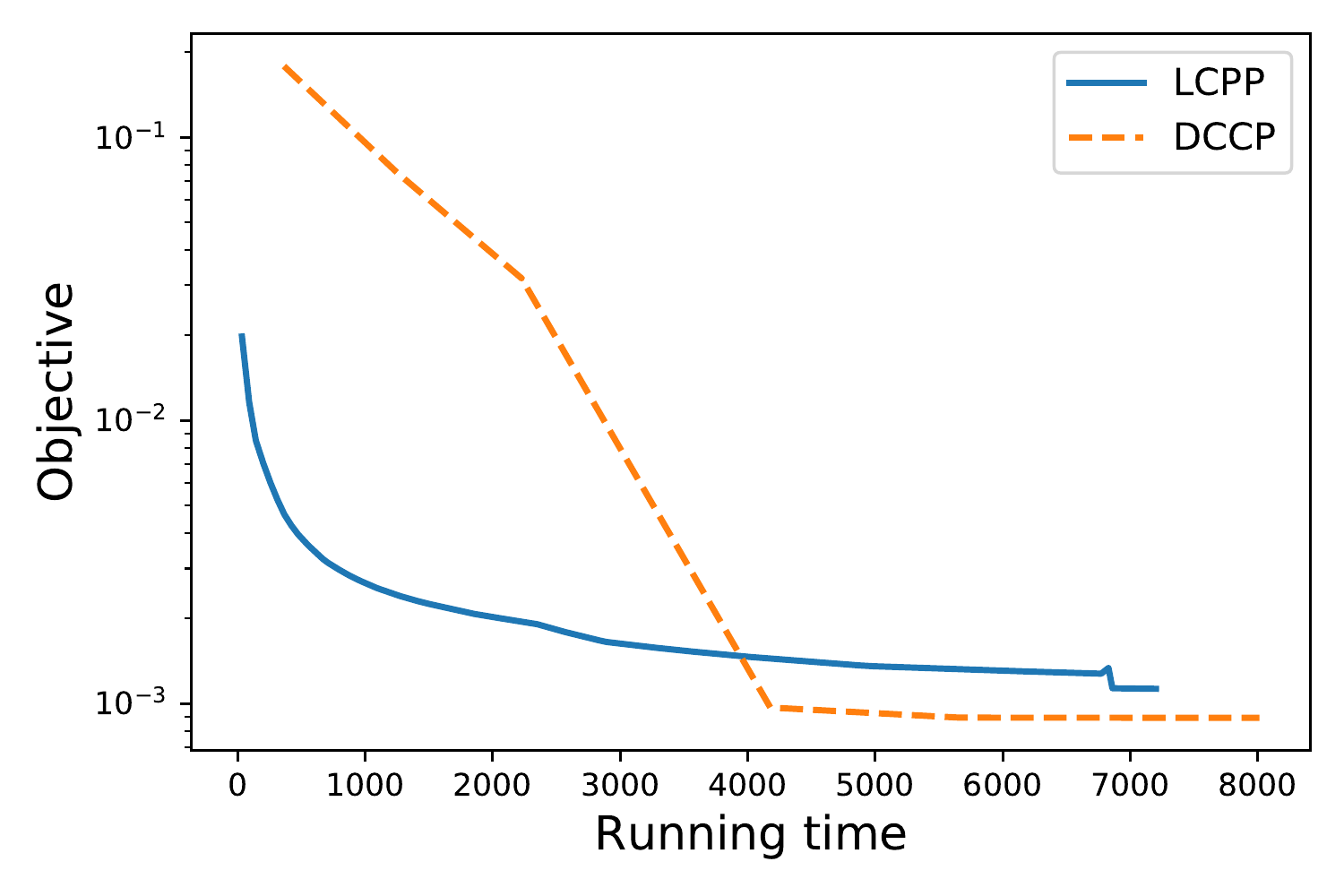}
	\includegraphics[scale=0.22]{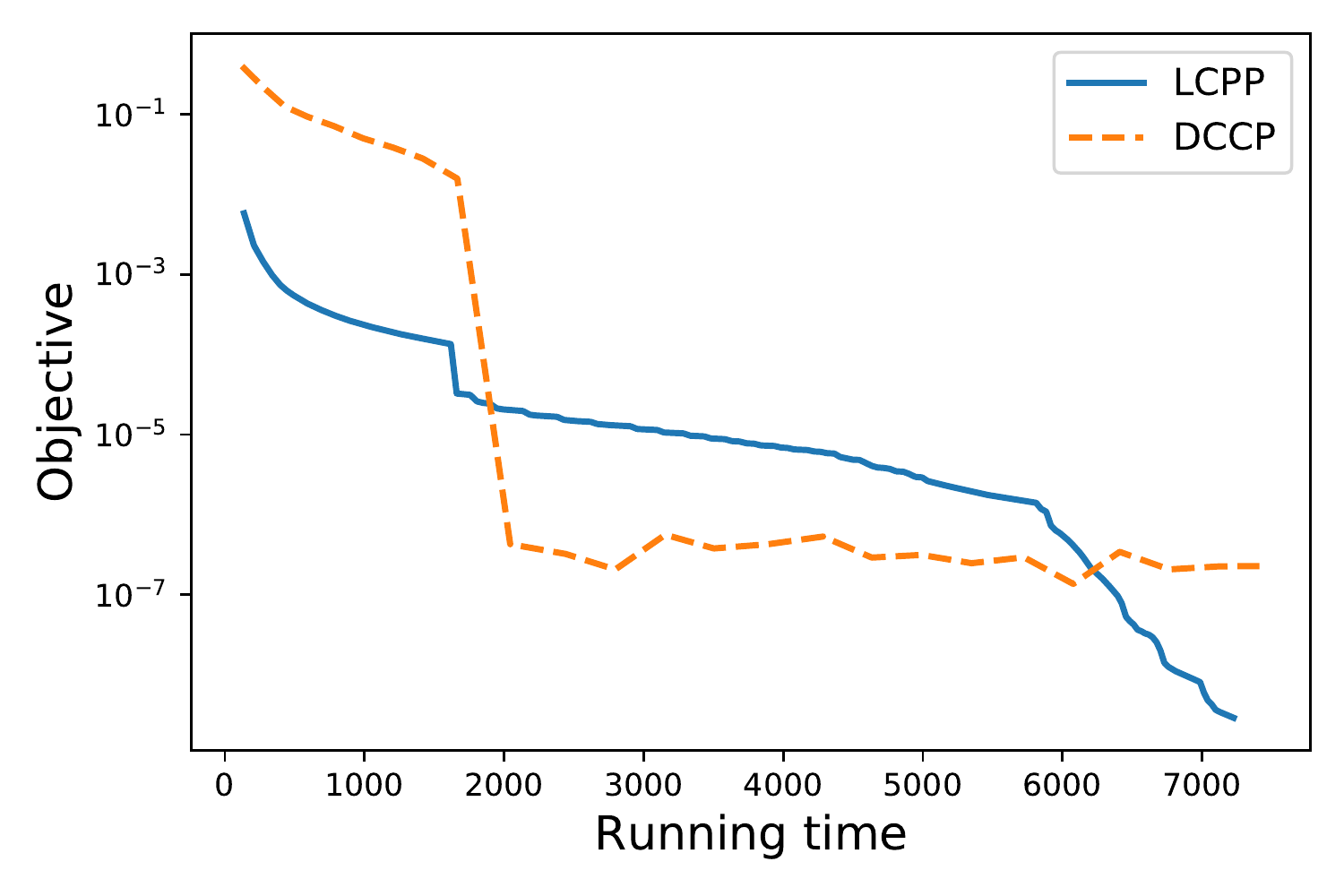}
	\caption{Objective value vs. running time (in seconds).  
		Left to right: \texttt{mnist} ($\eta=0.1d$), \texttt{real-sim} ($\eta=0.001d$),
		 \texttt{rcv1.binary} ($\eta=0.05d$) and \texttt{gisette} ($\eta=0.05d$). $d$ stands for the feature dimension. 
	}
	\label{fig:cmp-dccp}
\end{figure}
Comparison is conducted on the classification problem. To fix the parameters, we choose $\gamma=10^{-5}$ for \texttt{gisette} dataset and $\gamma=10^{-4}$ for the other datasets. For each LCPP subproblem we run gradient descent  at most 10 iterations and break when the criterion $\Vert x^k-x^{k-1}\Vert/\Vert x^k\Vert\le \vep$ is met. We set the number of outer loops as 1000 to run LCPP sufficiently long. We set $\lambda=2,\theta=0.25$ in the MCP function. 
Figure~\ref{fig:cmp-dccp} plots the convergence performance of LCPP and DCCP, confirming that LCPP is more advantageous over DCCP. 
Specifically, LCPP outperforms DCCP,  sometimes reaching near-optimality even before DCCP finishes the first iteration.
This observation can be explained by the fact that LCPP leverages the strengthen of first order methods, for which we can derive efficient projection subroutine. 
In contrast, DCCP is not scalable to large dataset due to the inefficiency in dealing with large scale linear system arising from the interior point subproblems.

Our next experiment is to compare the performance of nonconvex sparse constrained models, which is then optimized by LCPP, against regularized learning models in the following form:
\[\min_x\, \psi(x)=\tfrac{1}{n}\tsum_{i=1}^n L_i(x) + \alpha g(x).\]
As described  above, $g(x)$ is the sparsity-inducing penalty function and $L_i(x)$ is  a loss function on the data. We consider both convex and nonconvex penalties, namely Lasso-type penalty $g(x)=\Vert x\Vert_1$ and MCP penalty (see Table~\ref{tab:constraint_fun_examples}).
We solve the Lasso penalty problem by linear models  provided by Sklearn \cite{scikit-learn}   and solve the MCP regularized problem by the popular solver GIST  \cite{gong2013a}. For simplicity, both GIST and LCPP set $\lambda=2$ and $\theta=5$ in MCP function, and set the maximum iteration number as 2000 for all the algorithms.  Then we use a grid of values $\alpha$ for GIST and LASSO, and $\eta$ for LCPP accordingly, to obtain the testing error under  various sparsity levels. In Figure~\ref{fig:nnz-err} we report the 0-1 error for classification and mean squared error for regression. We can clearly see the advantage of our proposed models over Lasso-type estimators. We observe that nonconvex models LCPP and GIST both perform more robustly than Lasso  across a wide range of sparsity levels. Lasso models tend to overfit with increasing number of selected features while LCPP appears to be less affected by the feature selection. 
\begin{figure}[h]
	\includegraphics[scale=0.28]{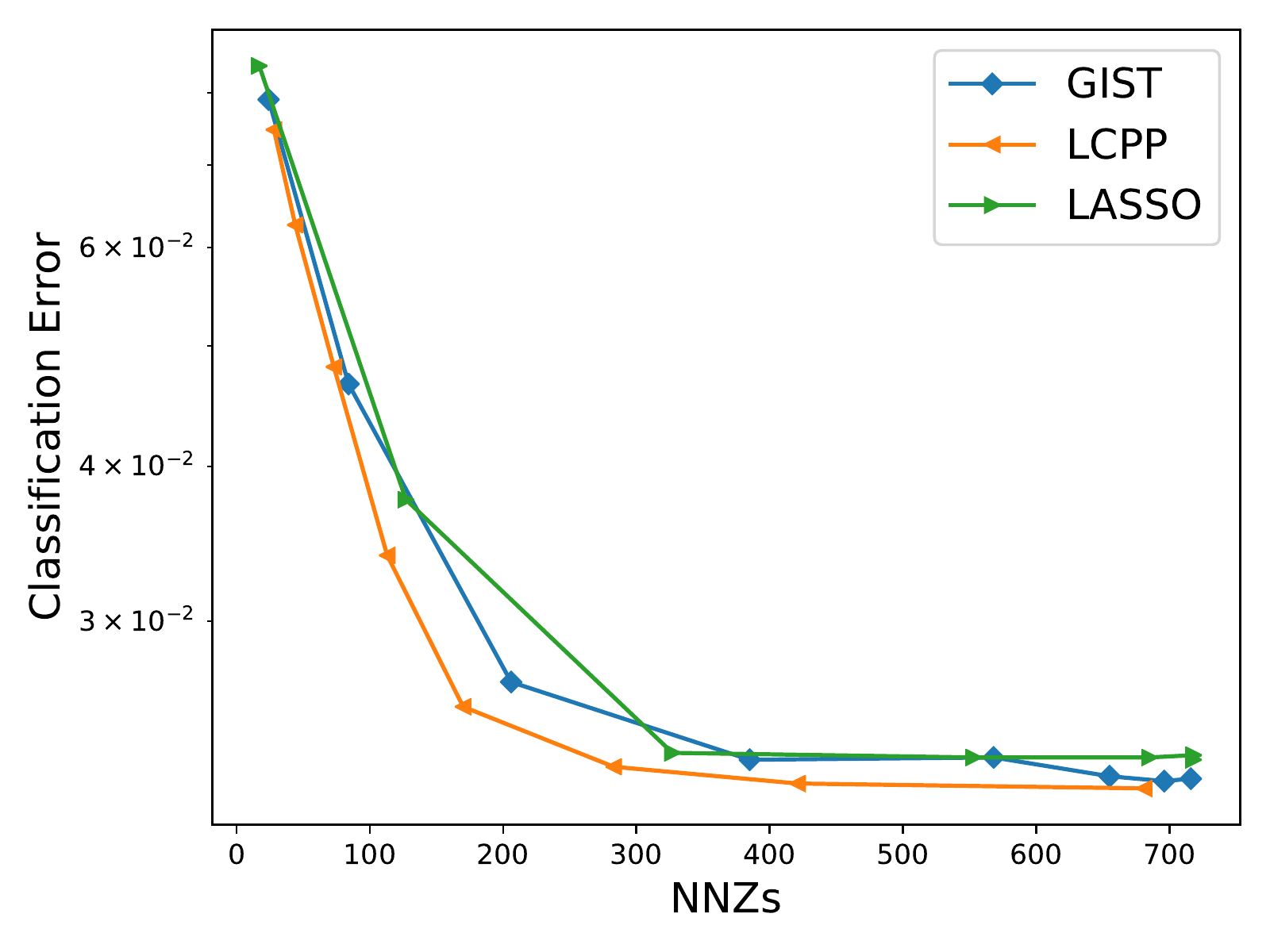} \includegraphics[scale=0.28]{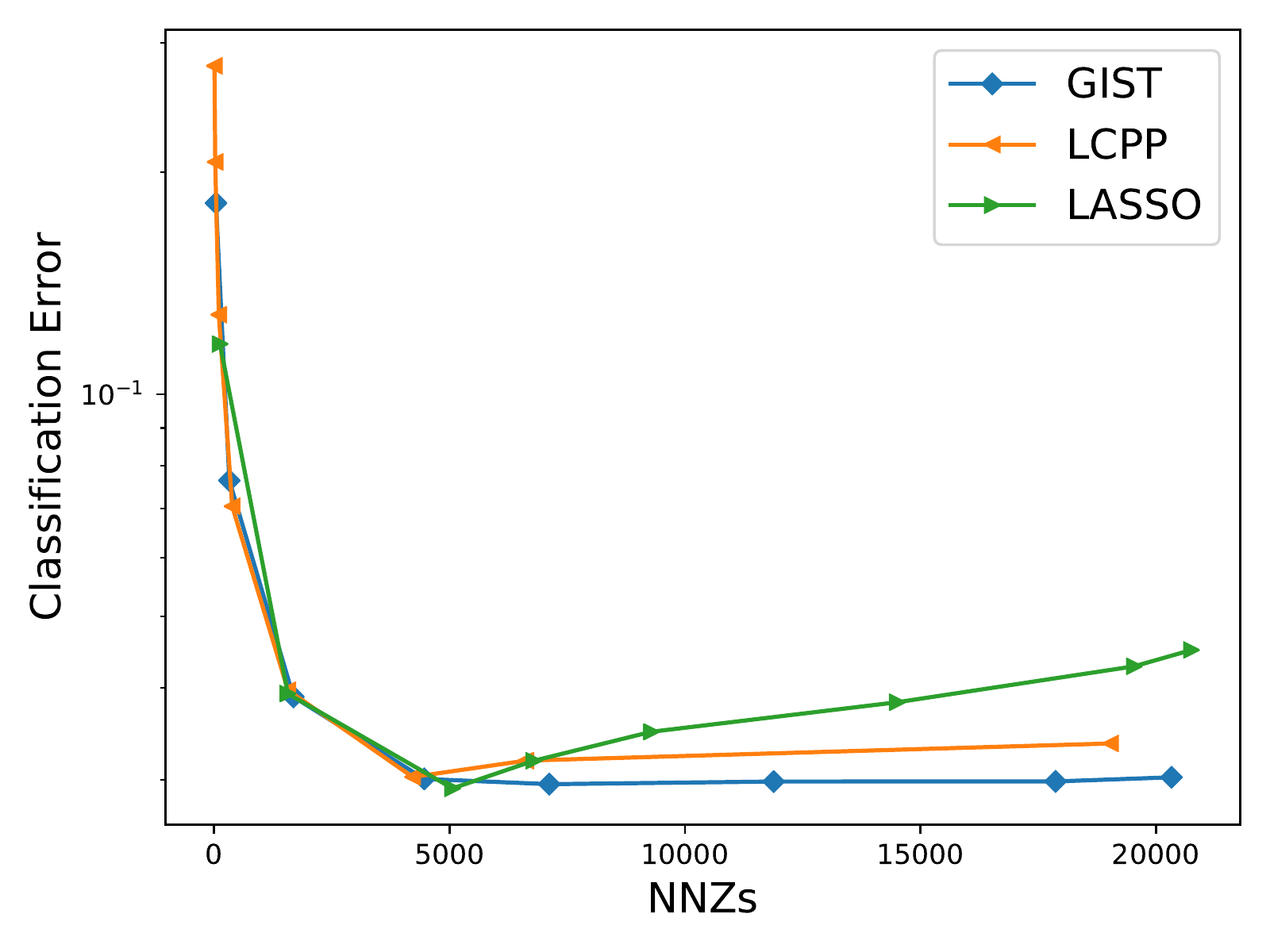}
	\includegraphics[scale=0.28]{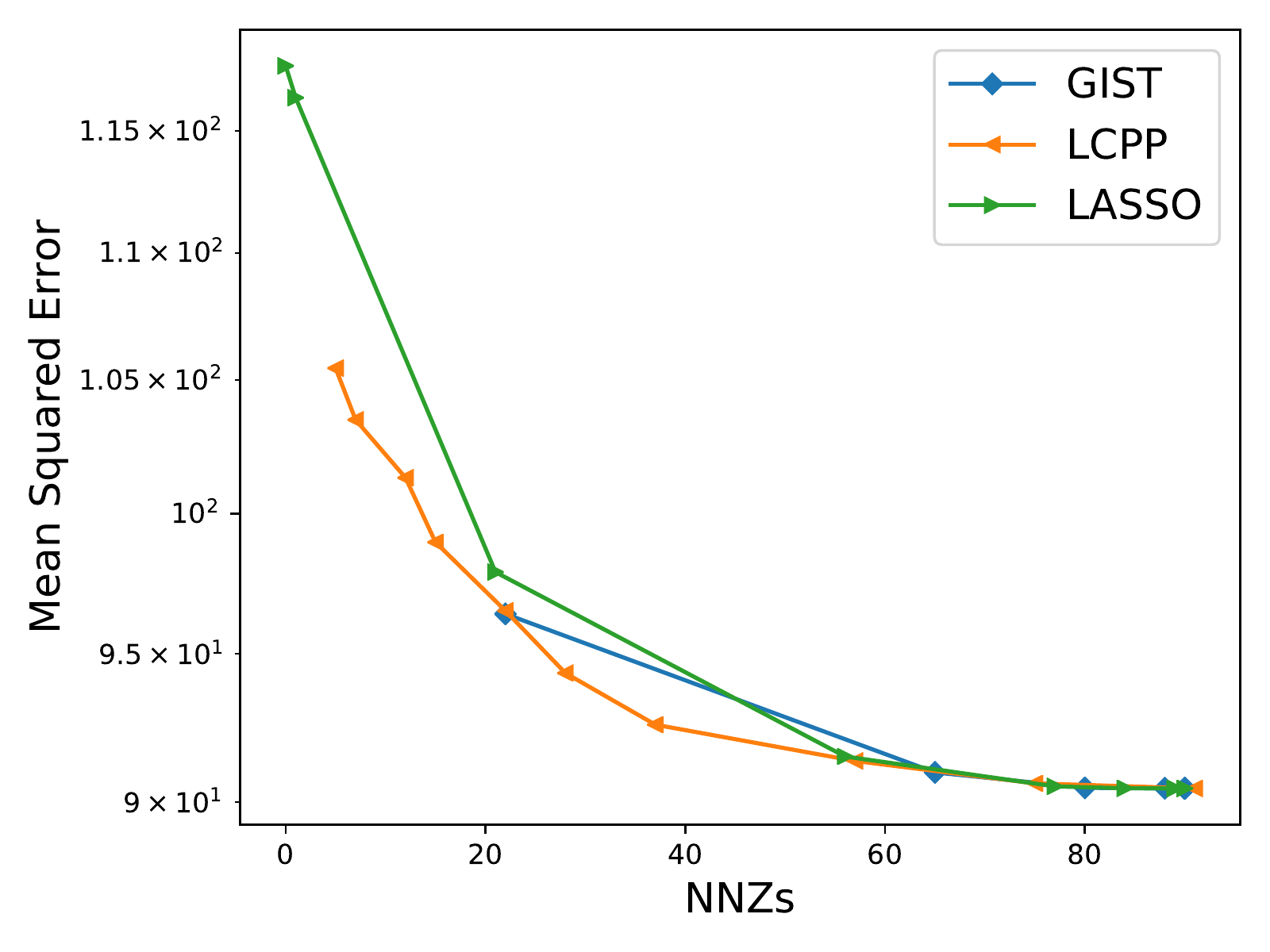}
	
	\includegraphics[scale=0.28]{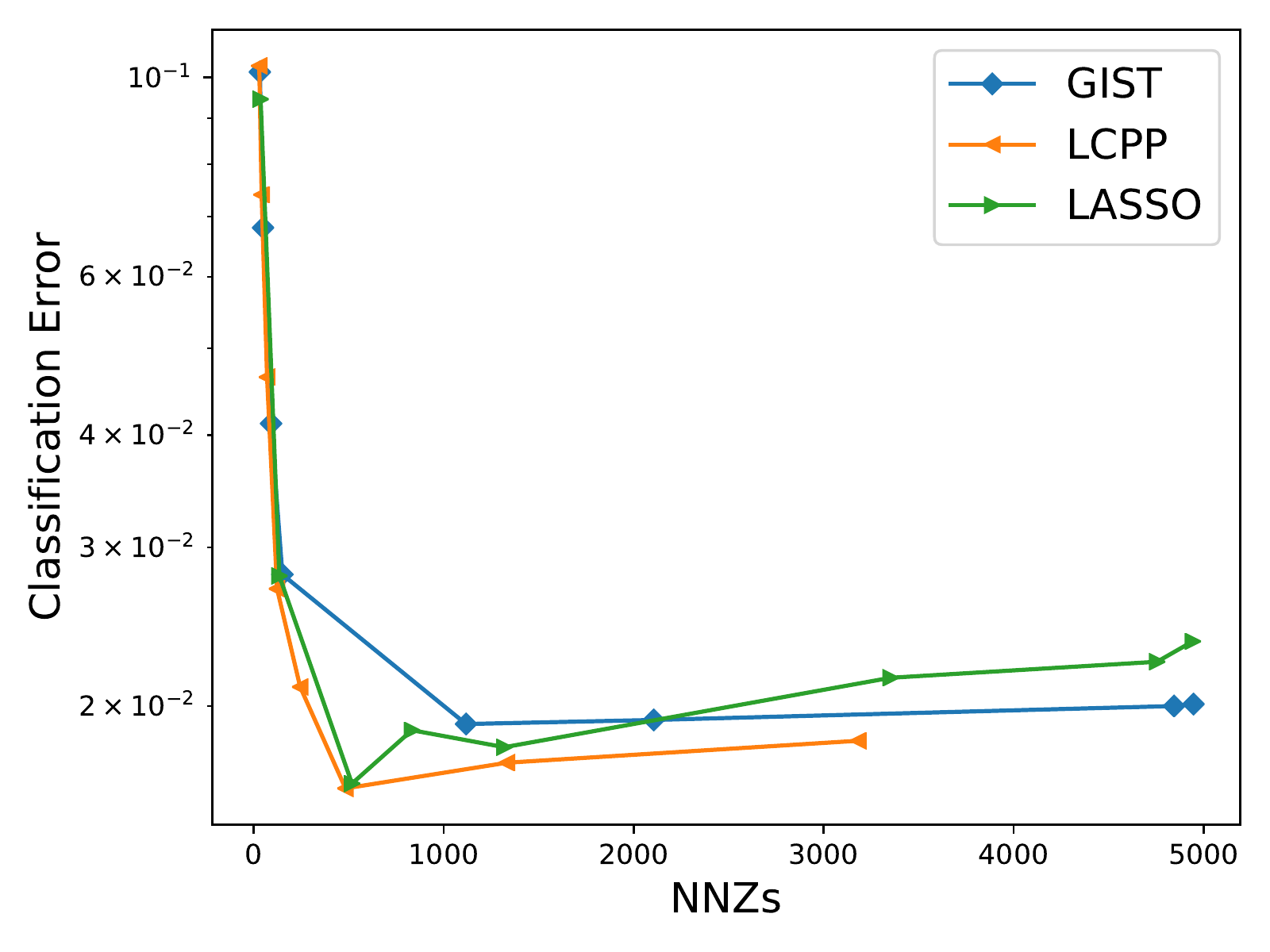}
	\includegraphics[scale=0.28]{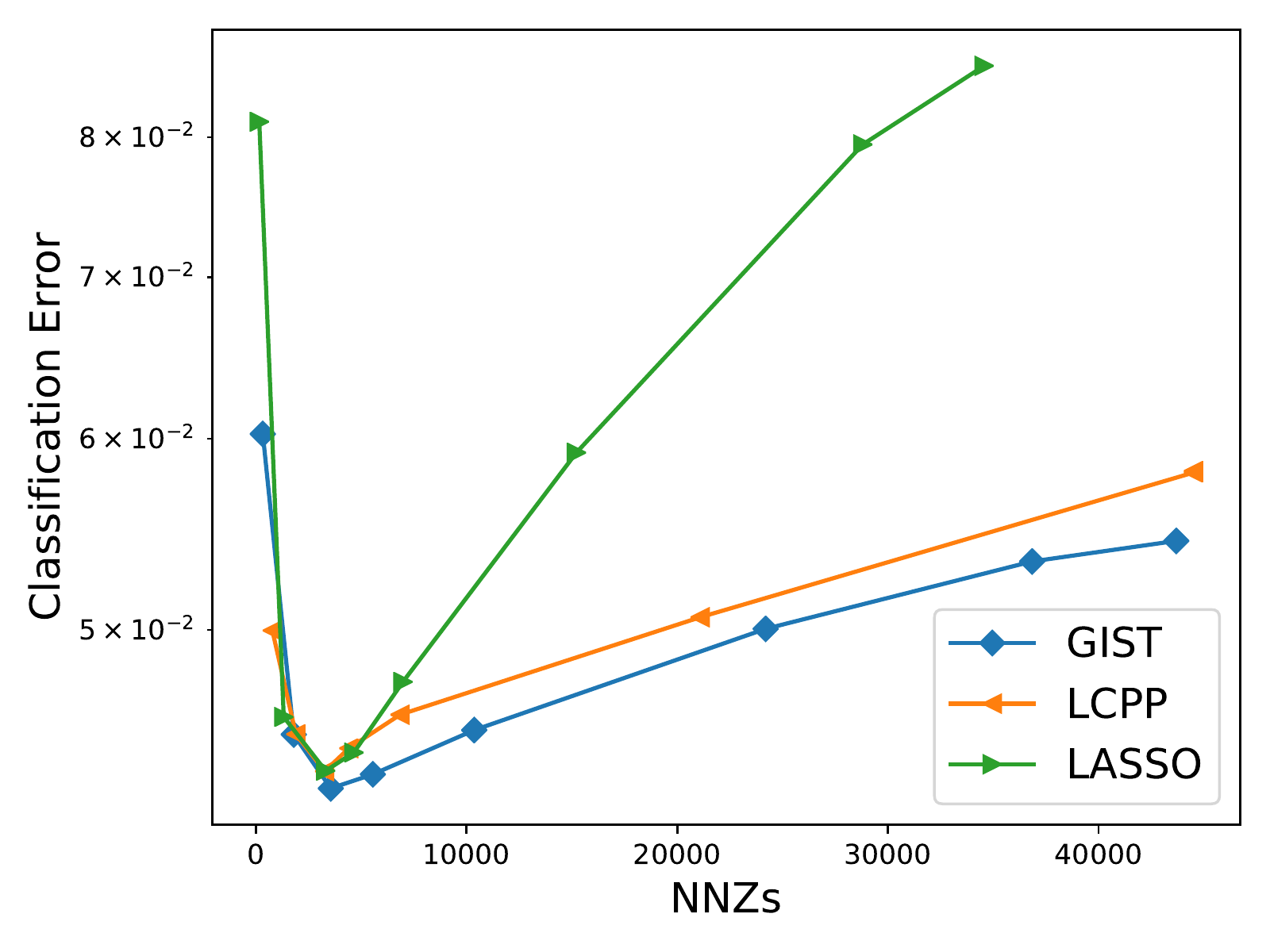}
	\includegraphics[scale=0.28]{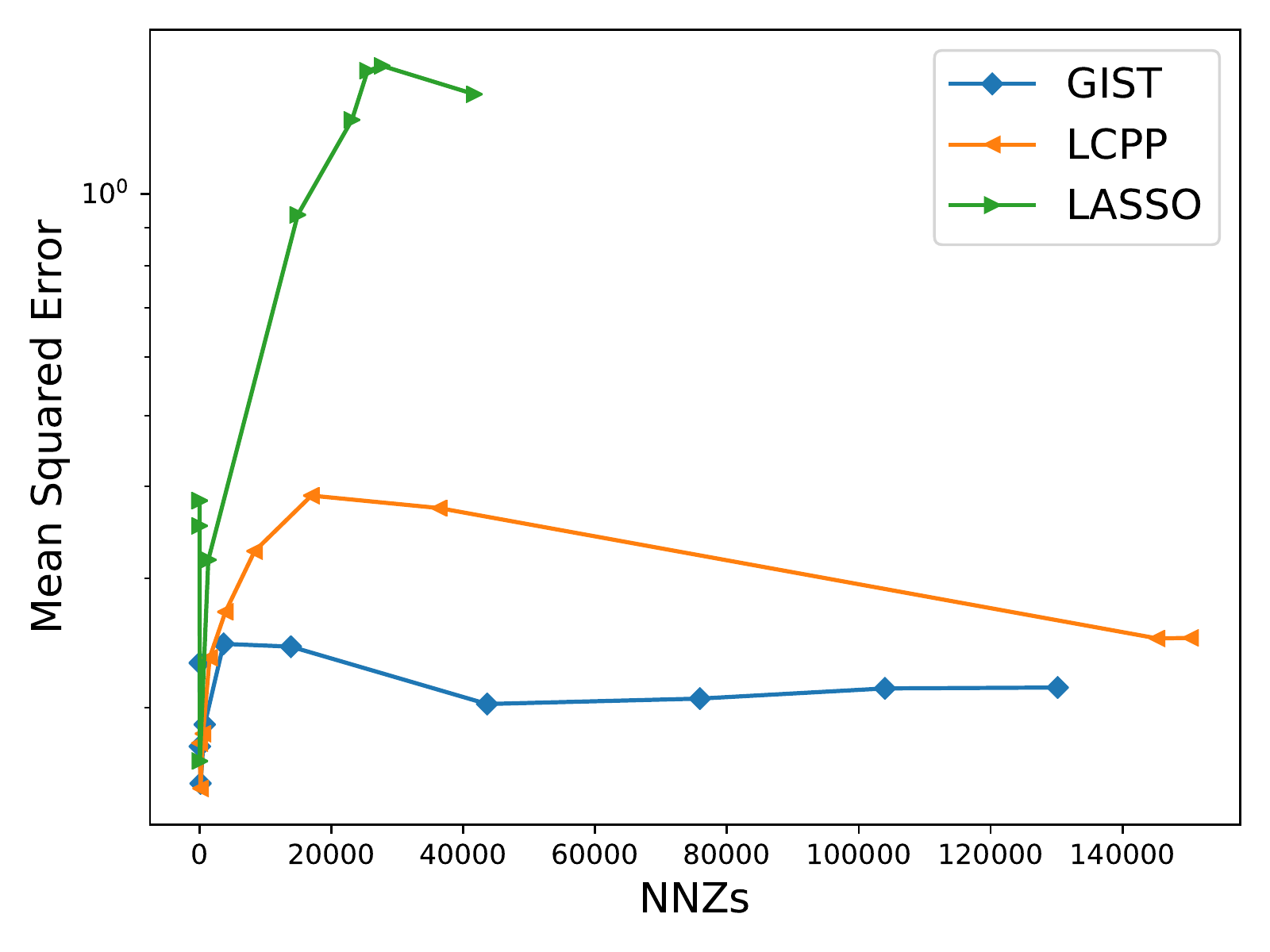} 
	
	\caption{\label{fig:nnz-err}Testing error vs number of nonzeros.
	First two columns show classification performance in clockwise order: \texttt{mnist}, \texttt{real-sim}, \texttt{rcv1.binary} and \texttt{gisette}. The third column shows regression test on \texttt{YearPredictionMSD} (top) and \texttt{E2006} (bottom).}
\end{figure}

\section{Conclusion}\label{sec:Conclusion}
We present a novel proximal point algorithm (LCPP) for nonconvex optimization with a nonconvex sparsity-inducing constraint.
We  prove the asymptotic convergence of the proposed algorithm to KKT solutions under mild conditions.
For practical use, we develop an efficient procedure for projection onto the subproblem constraint set, thereby adapting projected first order methods to LCPP for large-scale optimization
and establish an $\mathcal{O}({1}/{\vep})$$(\mathcal{O}({1}/{\vep^2}))$ complexity for deterministic (stochastic) optimization.
Finally, we perform numerical experiments to demonstrate the efficiency of our proposed algorithm for large scale sparse learning.

\section*{Broader Impact}
This paper presents a new model for sparse optimization and performs an algorithmic study for the proposed model. A rigorous statistical study of this model is still missing. We believe this was due to the tacit assumption that constrained optimization was more challenging compared to regularized optimization. This work takes the first step in showing that efficient algorithms can be developed for the constrained model as well.
Contributions made in this paper has the potential to inspire new research from statistical, algorithmic as well as experimental point of view in the wider sparse optimization area.

\section*{Acknowledgments}
Most of this work was done while Boob was at Georgia Tech. Boob and Lan gratefully acknowledge the National Science Foundation (NSF) for its support through grant CCF 1909298. Q. Deng acknowledges funding from National Natural Science Foundation of China (Grant 11831002).
\bibliographystyle{plain}
\bibliography{example_paper}

\newpage

\appendix
\section{Auxiliary results}
\subsection{Existence of KKT points}
\begin{prop}
	\label{prop:KKT-subseq}Under Assumption \ref{assu:feasibility-level},
	let $x^{0}=\hat{x}$. 
	Then, for any $k\ge 1$, we have $x^{k-1}$ is strictly feasible for the $k$-th
	subproblem. Moreover, there exists $\bar{x}^{k},\bar{y}^{k} \ge 0$
	such that $	g_{k}\left(\bar{x}^{k}\right) \le\eta_{k}$ and:
	\begin{equation}
	\begin{aligned}\partial\psi(\bar{x}^{k})+\gamma\left(\bar{x}^{k}-x^{k-1}\right)+\bar{y}^{k}\left(\partial g_{k}(\bar{x}^{k})\right) & \ni0\\
	\bar{y}^{k}\left(g_{k}\left(\bar{x}^{k}\right)-\eta_{k}\right) & =0
	\end{aligned}
	\label{eq:subprob-kkt}
	\end{equation}
\end{prop}
\begin{proof}
	Since $x^0$ satisfies $g(x^0) \le \eta_{0} < \eta_{1}$ so we have that first subproblem is well defined. We prove the result by induction. First of all, 
	suppose $x^{k-1}$
	is strictly feasible for $k$-th subproblem: $g_{k}(x^{k-1}) < \eta_{k}$. Then we note that this problem is also valid and a feasible $x^k$ exists. Hence, algorithm is well-defined. Now, note that 
	\[
	g_{k+1}(x^{k})= g(x^{k})\le g_{k}(x^{k})\le\eta_{k}<\eta_{k+1}.
	\]
	where first inequality follows due to majorization, second inequality follows due to feasibility of $x^k$ for $k$-the subproblem and third strict inequality follows due to strictly increasing nature of sequence $\{\eta_{k}\}$. \\
	Since $k$-th subproblem has $x^{k-1}$ as strictly feasible point satisfying Slater condition so we obtain existence of $\bar{x}^k$ and $\bar{y}^k \ge 0$ satisfying the KKT condition \eqref{eq:subprob-kkt}.
\end{proof}

\subsection{Proof of Theorem \ref{thm:conv_asym_KKT}} \label{apx:proof_asym_conv}
In order to prove this theorem, we first state the following intermediate result.
\begin{prop}
	\label{prop:square-sum-2}
	Let $\pi_k$ denotes the randomness of $x^1,x^2,...,x^{k-1}$. Assume that there exists a  $\rho\in[0, \gamma-\mu]$  and a summable nonnegative sequence $\zeta_k $ ($\zeta_k\ge 0,\, \tsum_{k=1}^{\infty}\zeta_{k}<\infty$) such that
	\begin{equation} 
	\Ebb\left[\psi_{k}(x^{k})-\psi_{k}(\bar{x}^{k})|\pi_k\right]\le \tfrac{\rho}{2}\norm{\bar{x}^k-x^{k-1}}^2 + \zeta_{k}
	\end{equation}
	Then, under Assumption \ref{assu:feasibility-level}, we have \\
	1. The sequence $\Ebb[\psi(x^k)]$ is bounded;\\
	2. $\lim_{k\raw\infty}\psi(x^{k})$ exists a.s.;\\
	3. $\lim_{k\raw\infty}\norm{x^{k-1}-\bar{x}^{k}}=0$ a.s.;\\
	4. If the whole algorithm is deterministic then $\psi(x^k)$ is bounded. Moreover, if $\zeta_k=0$, then the sequence $\psi(x^k)$ is monotonically decreasing and convergent.
\end{prop}
\begin{proof}
		Due to the strong convexity of $\psi_{k}(x)$, we have 
	\begin{align}
	\psi_{k}(\bar{x}^{k}) & \le\psi_{k}(x)-\tfrac{\gamma-\mu}{2}\norm{\bar{x}^{k}-x}^{2},\label{eq:mid2-optimality}
	\end{align}
	for all $x$ satisfying $g_k(x) \le \eta_{k}$.
	Taking $x=x^{k-1}$ and using feasibility of $x^{k-1}$ ($g_{k}(x^{k-1})\le\eta_{k})$
	we have 
	\[
	\psi(x^{k-1})\ge\psi(\bar{x}^{k})+\tfrac{\gamma}{2}\norm{\bar{x}^{k}-x^{k-1}}^{2}+\tfrac{\gamma-\mu}{2}\norm{x^{k-1}-\bar{x}^{k}}^{2}
	\]
	Together with \eqref{inexactness} we have 
	\begin{equation}\label{eq:middle-sum}
	\begin{split}
	\zeta_{k}+\psi(x^{k-1})&\ge\Ebb\big[\psi(x^{k})+\tfrac{\gamma}{2}\norm{x^{k}-x^{k-1}}^{2}|\pi_k\big]\\
	&+\tfrac{\gamma-\mu-\rho}{2}\norm{x^{k-1}-\bar{x}^{k}}^{2}.
	\end{split}
	\end{equation}
	Since $\{\zeta_k\}$ is summable, taking the expectation of $\pi_k$ and summing up all over all $k$, we have
	$\Ebb[\psi(x^k)]\le \psi(x^0) + \tsum_{s=1}^k \zeta_k<\infty$.
	Moreover, Applying Supermartingale Theorem \ref{thm:supermartingale} to \eqref{eq:middle-sum}, we have
	$\lim_{k\raw\infty}\psi(x^{k})$ exists and $\tsum_{k=1}^{\infty}\norm{x^{k-1}-\bar{x}^{k}}^{2}<\infty$
	a.s. Hence we conclude $\lim_{k\raw\infty}\norm{x^{k-1}-\bar{x}^{k}}=0$
	a.s. Part 4) can be readily deduced from \eqref{eq:middle-sum}.
\end{proof}
Now we are ready to prove Theorem \ref{thm:conv_asym_KKT}.

For simplicity, we assume the whole sequence generated by Algorithm \ref{alg:main} converges to $\wtil{x}$.
	Due to Proposition \ref{prop:KKT-subseq}, there exists a KKT point
	$(\bar{x}^{k},\bar{y}^{k}$). The optimality condition yields 
	\begin{equation}\label{eq:mid-opt-1}
	\begin{split}
	&\psi(x)+\frac{\gamma}{2}\norm{x-x^{k-1}}^{2}+\bar{y}^{k}g_{k}(x)\ge\psi(\bar{x}^{k})+\frac{\gamma}{2}\norm{\bar{x}^{k}-x^{k-1}}^{2}+\bar{y}^{k}g_{k}(\bar{x}^{k}),\quad\forall x
	\end{split}
	\end{equation}
	Since $\bar{y}^{k}$ is bounded, there exists a convergent subsequence
	$\{i_{k}\}$ that $\lim_{k\raw\infty}\bar{y}^{i_{k}}=\wtil y$ for
	some $\wtil y\ge0$. Let us take $k\raw\infty$ in (\ref{eq:mid-opt-1}).
	In view of Proposition \ref{prop:square-sum-2}, Part 3, we have $\lim_{k\raw\infty}\bar{x}^{i_{k}}=\lim_{k\raw\infty}x^{i_{k}-1}=\wtil x$ almost surely.
	Then $\lim_{k\raw\infty}h(x^{i_{k}-1})=h(\wtil x)$ and $\lim_{k\raw\infty}\nabla h(x^{i_{k}-1})=\nabla h(\wtil x)$ a.s.
	due to the continuity of $h(x)$ and $\nabla h(x)$, respectively.
	Then we have 
	\begin{align*}
	&\psi(x)+\frac{\gamma}{2}\norm{x-\wtil x}^{2}+\wtil y\left[\lambda\norm x_{1}-h(\wtil x)-\inprod{\nabla h(\wtil x)}{x-\wtil x}\right] \ge\psi(\wtil x)+\wtil yg(\wtil x), \quad a.s.
	\end{align*}
	implying that $\wtil x$ minimizes the loss function $\psi(x)+\frac{\gamma}{2}\norm{x-\wtil x}^{2}+\wtil y\left[\lambda\norm x_{1}-h(\wtil x)-\inprod{\nabla h(\wtil x)}{x-\wtil x}\right]$.
	Due to the first order optimality condition, we conclude $0\in\partial\psi(\wtil x)+\wtil y\partial g(\wtil x),$ a.s.
	
	Moreover, using the complementary slackness, we have $0=\bar{y}^{i_{k}}\left(g_{i_{k}}\left(\bar{x}^{i_{k}}\right)-\eta_{i_{k}}\right)$.
	Taking the limit of $k\raw\infty$ and noticing that $\lim_{k\raw\infty}\eta_{i_{k}}=\eta$,
	we have $0=\wtil y\left(g\left(\wtil x\right)-\eta\right)$ a.s	. 
	As a result, we conclude that $(\wtil x,\wtil y)$ is a KKT point of problem
	\eqref{noncvx-constraint}, a.s.

\subsection{Proof of Theorem \ref{thm:bound_dual_MFCQ_new}}
From KKT condition of (\ref{eq:subprob-kkt}), $\bar{x}^{k}$ is the
optimal solution of the problem
$
\min_{x\in\Rbb^d}\,\psi_{k}(x)+\bar{y}^{k}\left(g_{k}(x)-\eta_{k}\right).
$
Therefore, for any $x\in\Rbb^d$, we have 
\begin{align}\label{eq:middle-optimality-1_new}
\psi_{k}(x)+\bar{y}^{k}g_{k}(x) \ge\psi_{k}(\bar{x}^{k})+\bar{y}^{k}g_{k}(\bar{x}^{k})
\end{align}

We prove that $\{\bar{y}^{k}\}$ is bounded a.s. by contradiction. If $\left\{ \bar{y}^{k}\right\} $
has unbounded subsequence with positive probability, then conditioned under that event, there exists a subsequence $\{i_{k}\}$
such that $\bar{y}^{i_{k}}\raw\infty$. Let us divide both sides of
(\ref{eq:middle-optimality-1_new}) by $\bar{y}^{k}$ and expand $g_{k}$
by its definition. After placing $k=i_{k}$
, we have for all $x$
\begin{equation}\label{eq:mid-opt-2-1_new}
\begin{split}
&\tfrac{1}{\bar{y}^{i_{k}}}\psi_{i_{k}}(x)+\lambda\norm x_{1}-\nabla h(x^{i_{k}-1})^Tx\\
&\ge\tfrac{1}{\bar{y}^{i_{k}}}\psi_{i_{k}}(\bar{x}^{i_{k}})+\lambda\norm{\bar{x}^{i_{k}}}_{1}-\nabla h(x^{i_{k}-1})^T\bar{x}^{i_{k}}.
\end{split}
\end{equation}
Let $\wtil x$ be any limiting point a.s. of the sequence $\left\{ x^{i_{k}-1}\right\} $. \added{By the statement of the theorem, we know that it exists and satisfies MFCQ assumption.}
Passing to some subsequence if necessary, we have $\lim_{k\raw\infty}x^{i_{k}-1}=\wtil x$ a.s.
Using Proposition \ref{prop:square-sum-2} Part 3, we have $\lim_{k\raw\infty}\bar{x}^{i_{k}}=\wtil x$ a.s. Moreover, using Proposition \ref{prop:square-sum-2} Part 2, we have $\lim_{k\raw\infty} \psi(\wb{x}^{i_k})$ exists a.s. This implies $\lim_{k\raw\infty} \tfrac{1}{\wb{y}^{i_k}} \psi_{i_{k}}(\wb{x}^{i_k}) = 0$ a.s.\\
Taking $k\raw\infty$, since $\psi_{i_{k}}(x)$ is bounded a.s. (due to existence of $\wt{x}$ a.s.), we have
$\lim_{k\raw\infty}\frac{1}{\bar{y}^{i_{k}}}\psi_{i_{k}}(x)=0$.
From Lipschitz continuity of $l_{1}$ norm and $\nabla h(x)$, we have $\lim_{k\raw\infty}\lambda\norm{\bar{x}^{i_k}}_{1}=\lambda\norm{\wtil x}_{1}$ a.s.,
and $\lim_{k\raw\infty}\nabla h(x^{i_k-1})=\nabla h(\wtil x)$ a.s., respectively. It then follows from (\ref{eq:mid-opt-2-1_new}) that for all $x$, we have 
$\lambda\norm x_{1}-\inprod{\nabla h(\wtil x)}x\ge\lambda\norm{\wtil x}_{1}-\inprod{\nabla h(\wtil x)}{\wtil x}.$ In other words, we have 
\begin{equation}\label{eq:mid-opt-3-1_new}
\mathbf{0}\in\partial\lambda\norm{\wtil x}_{1}-\nabla h(\wtil x)=\partial g(\wtil x), a.s.
\end{equation}
Moreover, due to complementary slackness and $\bar{y}^{i_{k}}>0$,
the equality $g_{i_{k}}(\bar{x}^{i_{k}})=\eta_{i_{k}}$ holds. Hence, in the limit, we have the constraint $g(\wtil x)=\eta$ active a.s. Under
MFCQ, there exists $z$ such that $\max_{v\in\partial g(\wtil{x})} z^Tv<0$.
However, from
(\ref{eq:mid-opt-3-1_new}) we have $0=z^{T}\mathbf{0}$ since $\mathbf{0} \in \partial g(\wt{x})$, leading
to a contradiction to the event that $\{\wb{y}^k\}$ contained unbounded sequence with positive probability. Hence, $\wb{y}$ is bounded a.s.

\section{Explicit and specialized bounds on the dual}\label{apx:bound_dual_overall}
Here, we discuss some of the results for explicit bounds on the dual. In particular, we focus on the SCAD and MCP case. Similar results can be extended for Exp and $\ell_p, p<0$ case since these function follows two key properties (as we will see later in the proofs):
\begin{enumerate}
	\item $\abs{\grad h(x)} \le \lambda$ for all $x$ for each of these functions.
	\item They remain bounded below a constant. See Figure \ref{fig:constraint_plots}.
\end{enumerate}
We exploit these two structural properties of these sparse constraints to obtain specialized and explicit bounds on the optimal dual of problem \ref{noncvx-constraint}. The following lemma is in order.
	
	\begin{lem}
		\label{lem:bound_dual_value}
		Let $h:\Rbb\to \Rbb$ be the the convex function which 
		satisfies $\abs{\nabla h(x)} \le \lambda$ for all $x \in \Rbb$. Then the minimum value of $\wb{g}(x; \wb{x}) :\Rbb\to\Rbb$ defined as $\wb{g}(x; \wb{x}) := \lambda \abs{x} - h(\wb{x}) - \inprod{\nabla h(\wb{x})}{x-\wb{x}}$ is achieved at $0$ for all $\wb{x} \in \Rbb$.
	\end{lem}
\begin{proof}
	Note that $\wb{g}$ is a convex function for any $\wb{x} \in \Rbb$. So by first order optimality condition, if $\wh{x}$ is the minimizer of $\wb{g}$ then $0 \in \partial \wb{g}(\wh{x}; \wb{x})$. This implies
	\[\lambda \partial \abs{\wh{x}} - \nabla h(\wb{x}) \ni 0.\]
	Note that $\wh{x} = 0$ satisfies this condition since in that case $\lambda \partial \abs{\wh{x}} = [-\lambda, \lambda]$. And due to assumption on $h$, we have $\nabla h(\wb{x}) \in  [-\lambda, \lambda]$. Hence $\wh{x} = 0$ is always the minimizer.
\end{proof}
	Now note that $h_{\lambda, \theta}$ functions defined for our examples, such as SCAD or MCP. satisfy the assumption of bounded gradients in Lemma \ref{lem:bound_dual_value}. Now we use this simple result to show that $\mathbf{0}$ is the most feasible solution for each of the subproblem \eqref{subprob} generated in Algorithm \ref{alg:main} and hence we can give an explicit bound for the optimal dual value for each subproblem.
	\begin{lem}
		\label{lem:dual_bound}
		Suppose all assumptions in Lemma \ref{lem:bound_dual_value} are satisfied. Then we have for any $k \ge 1$, 
		\begin{equation} 
		\label{eq:bound_y_bar}
		\wb{y}^k \le  \tfrac{\psi_{k}(\zerobf)-\psi_{k}(\wb{x}^k)}{\eta_{k} - g(x^{k-1}) +\tsum_{i = 1}^d(\lambda-\abs{\nabla h(x^{k-1}_i)}) \abs{x^{k-1}_i}}.
		\end{equation}
	\end{lem}
\begin{proof}
	Note that $g_k(x) = \tsum_{i = 1}^{d} \wb{g}(x_i; x^{k-1}_i)$ where $\wb{g}$ is defined in Lemma \ref{lem:bound_dual_value}. Since assumptions of Lemma \ref{lem:bound_dual_value} hold, so we have that each individual $\wb{g}$ is minimized at $x_i = 0$. Hence $g_k(\mathbf{0})$ is the minimum value of $g_k$. In view of Proposition \ref{prop:KKT-subseq}, we have that $x_{k-1}$ is strictly feasible solution with respect to constraint $g_k(x) \le \eta_{k}$ implying $g_k(x^{k-1}) -\eta_{k} < 0$. Hence,  we have
	\begin{align*}
	&\eta_{k} -g_k(\mathbf{0})\\
	&= \eta_{k} - \big[\lambda \norm{\mathbf{0}}_1-\tsum_{i = 1}^{d}\{ h(x^{k-1}_i) + \nabla h(x^{k-1}_i)(0-x^{k-1}_i) \} \big]\\
	&= \eta_{k} + \tsum_{i = 1}^d h(x^{k-1}_i) - \tsum_{i = 1}^d\nabla h(x^{k-1}_i)x^{k-1}_i \\
	&= \eta_{k}  -g(x^{k-1}) + [g(x^{k-1}) + h(x^{k-1})] - \tsum_{i = 1}^d\nabla h(x^{k-1}_i)x^{k-1}_i \\
	&\ge \eta_{k} - g(x^{k-1}) + \lambda\norm{x^{k-1}}_1 - \tsum_{i=1}^d\abs{\nabla h(x^{k-1}_i)} \abs{x^{k-1}_i}\\
	&= \eta_{k} - g(x^{k-1}) + \tsum_{i=1}^d(\lambda-\abs{\nabla h(x^{k-1}_i)})\abs{x^{k-1}_i}\\
	&> 0.
	\end{align*}
	Here, last strict inequality follows due to the fact that $\lambda \ge \abs{\grad h(x^{k-1}_i)}$ and $\eta_{k} > g(x^{k-1})$.
	Then, we have, optimal dual $\wb{y}^k$ satisfies for all $x$:
	\begin{align*}
	\psi_{k}(\wb{x}^k) &\le \psi_{k}(x) + \wb{y}^k(g_k(x)-\eta_{k})\\
	\Rightarrow \psi_{k}(\wb{x}^k) &\le \psi_{k}(\mathbf{0}) + \wb{y}^k (g_k(\mathbf{0})-\eta_{k})\\
	\Rightarrow \wb{y}^k &\le \tfrac{\psi_{k}(\zerobf) - \psi_{k}(\wb{x}^k)}{\eta_{k} - g_k(\zerobf)}\\
	&\le \tfrac{\psi_{k}(\zerobf)-\psi_{k}(\wb{x}^k)}{\eta_{k} - g(x^{k-1}) +\tsum_{i = 1}^d(\lambda-\abs{\nabla h(x^{k-1}_i)}) \abs{x^{k-1}_i}},
	\end{align*}
	where third inequality follows due to the fact that $\eta_{k} - g_k(\mathbf{0}) > 0$
	Hence, we conclude the proof.
\end{proof}

	Note that the bound in \eqref{eq:bound_y_bar} depends on $x^{k-1}$ which can not be controlled, especially in the stochastic cases. In order to show a bound on $\wb{y}^k$ irrespective of $x^{k-1}$, we must lower bound the denominator in \eqref{eq:bound_y_bar} for all possible values of $x^{k-1}$. To accomplish this goal, we show the following two theorems in which we lower bound the term $\tsum_{i = 1}^d (\lambda-\abs{\grad h(x^{k-1}_i)}) \abs{x^{k-1}_i}$. Each of these theorem is a specialized result for SCAD and MCP function, respectively. 
	\begin{thm}
		\label{thm:dual_bd_value}
		Let $g$ be the SCAD function and $x\in \Rbb^d$ such that $g(x) =  \alpha$. Also, let $\gamma = \alpha - \beta \tfrac{\lambda^2(\theta+1)}{2}$ where $\beta$ is the largest nonnegative integer such that $\gamma \ge 0$. Then, $\tsum_{i=1}^d(\lambda - \abs{\grad h(x_i)}) \abs{x_i} \ge z(\gamma)$ where $z:[0, \tfrac{\lambda^2(\theta+1)}{2}] \to \Rbb_{\ge 0}$ is the function defined as 
		\[ z(\gamma) := \begin{cases}
		\gamma &\text{if } 0 \le \gamma \le \lambda^2\\
		\tfrac{\gamma}{\lambda}\sqrt{\tfrac{2}{\theta-1}}\sqrt{\tfrac{\lambda^2(\theta+1)}{2} - \gamma} & \text{if } \lambda^2 < \gamma \le \tfrac{\lambda^2(\theta+1)}{2}
		\end{cases}.\]
	\end{thm}
	
	\begin{thm}
		\label{thm:mcp_y_bd_val}
		Let $g$ be the MCP function and $x \in \Rbb^d$ be such that $g(x) = \alpha.$ Also let $\gamma = \alpha - \beta \tfrac{\lambda^2\theta}{2}$ where $\beta$ is the largest nonnegative integer such that $\gamma \ge 0$. Then $\tsum_{i=1}^d(\lambda-\abs{\grad h(x_i)})\abs{x_i} \ge z(\gamma)$ where $z : [0, \tfrac{\lambda^2\theta}{2}]\to \Rbb_{\ge 0}$ is the function defined as
		$ z(\gamma) := \gamma\sqrt{1-\tfrac{2\gamma}{\theta\lambda^2}} .$
	\end{thm}
	
	Note that Theorem \ref{thm:dual_bd_value} states that lower bound $z(\gamma) = 0$ when $\gamma = 0$ or $ \tfrac{\lambda^2(\theta+1)}{2}$. In essence, when $\alpha$ is exact integral multiple of $\tfrac{\lambda^2(\theta+1)}{2}$ then lower bound turn out to be zero. However, for all other values of $\alpha$, the corresponding $z(\gamma)$ is strictly positive. This can be seen from the graph of $z(\gamma)$ below. Similar claims can be made with respect to MCP in Theorem \ref{thm:mcp_y_bd_val}.
	\begin{figure}[t]
		\centering
		\includegraphics[width=0.4\linewidth]{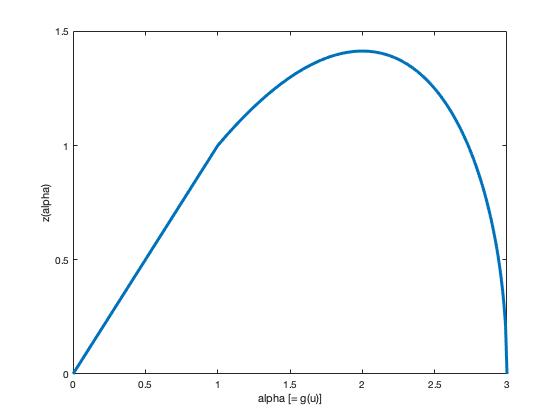}
		\caption{Plot of $z(\gamma)$ for SCAD function where $\lambda = 1, \ \theta = 5$. $z:[0,3]\to \Rbb_{\ge 0}$ where $z(0) = z(3) = 0$ otherwise $z$ is strictly positive.}
	\end{figure}
	
	Now we are ready to show a bound on $\wb{y}^k$ irrespective of $x^{k-1}$. We give a specific routine to choose the values of $\eta_k$ such that we can obtain a provable bound on the denominator in \eqref{eq:bound_y_bar} hence obtaining an upper bound on the $\wb{y}^k$ for all $k$ irrespective of $x^{k-1}$.
	\begin{prop}
		\label{prop:SCAD_eta_setting}
		Let $g$ be the SCAD function and $\eta = \beta\tfrac{\lambda^2(\theta+1)}{2} + \wt{\eta}$ where $\beta$ be the largest nonnegative integer such that $\wt{\eta} \ge 0$. Then, for properly selected $\eta_{0}$, we have that $\eta_{k} -g(x^{k-1}) + \tsum_{i = 1}^d(\lambda-\abs{\grad h(x^{k-1}_i)}) \abs{x^{k-1}_i} \ge \min\{\lambda^2, \tfrac{z(\wt{\eta})}{2}\}$.
	\end{prop}
	We note that very similar proposition for MCP can be proved based on Theorem \ref{thm:mcp_y_bd_val}. We skip that discussion in order to avoid repetition.
	
	\paragraph{Connection to MFCQ}
	In this section, we show the connection of MFCQ assumption in Theorem \ref{thm:bound_dual_MFCQ_new} with the bound in Theorem \ref{thm:dual_bd_value}. 
	
	Note that for the boundary points of the set $g(x) \le \eta_1$ where $\eta_{1} = \tfrac{\lambda^2(\theta+1)}{2}$ then the lower bound $z(\eta_1) = 0$. In fact, carefully following the proof of Theorem \ref{thm:dual_bd_value}, we can identify that the lower bound is tight for $x$'s such that one of the coordinate $x_i$ satisfy $\abs{x_i} \ge \lambda\theta$ and all other coordinates are $0$. In this case, we see that such points do not satisfy MFCQ. At such points, we don't have any strictly feasible directions required by MFCQ assumption. This can be easily visualized in the Figure \ref{fig:fig2} part (a) below. Note that $\lambda \theta =5$ and for any $\abs{x} \ge 5$, the feasible region is merely the axis and hence there is no strict feasible direction. This implies MFCQ indeed fails at these points.
	
	For $g(x) = \eta_2 < \eta_{1}$ the lower bound $z(\eta_2)$ is nonzero and same holds for $g(x) = \eta_3 > \eta_{1}$. Indeed, we see that for such cases, the points not satisfying MFCQ in case of $\eta_{1}$ vanish. This can be observed in Figure \ref{fig:fig2} part (b) and part (c). For the case of $\eta_{2}$ in part (b), these points become infeasible and for the case of $\eta_3$ in part (c), they are no longer boundary points. 
	
	Looking back at MFCQ from the result of Theorem \ref{thm:dual_bd_value}, we can see that how close $\eta$ is to $\tfrac{\lambda^2(\theta+1)}{2}$ shows how `close' the problem is for violating MFCQ. Moreover, the lower bound $z(\cdot)$ on the denominator of \eqref{eq:bound_y_bar} shows how quickly the dual will explode as the problem setting gets closer to violating MFCQ. 
	\begin{figure}[h]
		\centering
		\subfigure[$\eta_1 = 3$]{\includegraphics[width = 0.32\linewidth]{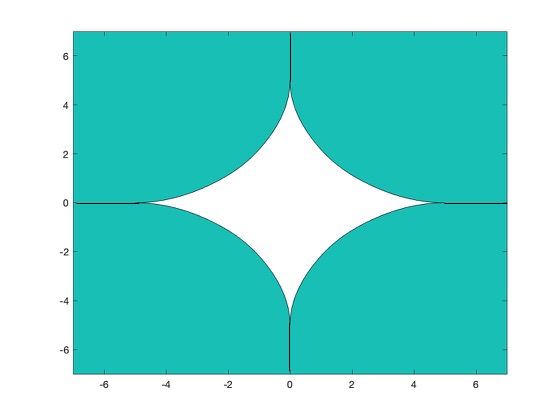}}
		\subfigure[$\eta_2 = 2.8$]{\includegraphics[width = 0.32\linewidth]{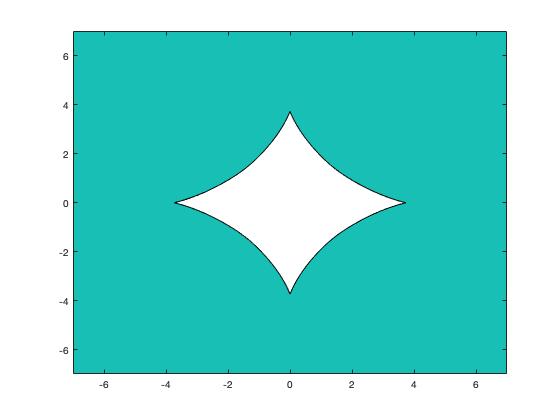}}
		\subfigure[$\eta_3 = 3.2$]{\includegraphics[width = 0.32\linewidth]{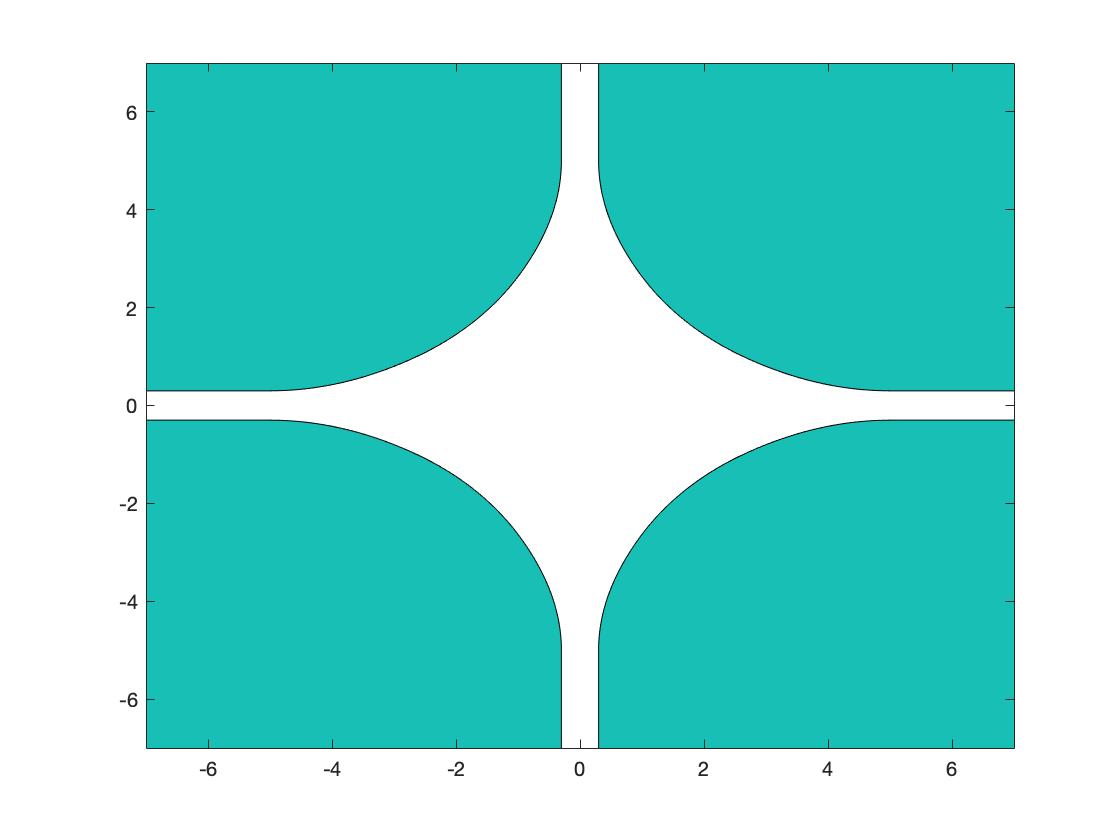}}	
		\caption{All figures are plotted for $\lambda = 1$ and $\theta  = 5$. Then $\eta_{1} = \tfrac{\lambda^2(\theta+1)}{2} = 3$. In fig (a), we see that for $\abs{x} \ge 5$, the MFCQ assumption is violated since only $x$-axis is feasible. Similar observation holds for $y$-axis as well. However, in fig(b) and fig(c) such claims are no longer valid.}
		\label{fig:fig2}
	\end{figure}\\
We complete	this discussion by showing the proof of Theorem \ref{thm:dual_bd_value} and Theorem \ref{thm:mcp_y_bd_val}. We also note that similar theorems can be proved for $\ell_p, p < 0$ and Exp function in Table \ref{tab:constraint_fun_examples}.
	
\subsection{Proof of Theorem \ref{thm:dual_bd_value}}\label{apx:dual_bd_value}
	First, we show a lower bound for one-dimensional function and then extend it to higher dimensions. Suppose $u \in \Rbb$ be such that $g(u) = \alpha$. Note that since $g$ is SCAD function so $\alpha$ must lie in the set $ [0, \tfrac{\lambda^2(\theta+1)}{2}]$. Key to our analysis is the lower bound on $(\lambda-\abs{\grad h(u)} ) \abs{u}$ as a function of $\alpha$. Note that since 
	\begin{equation}
	\label{eq:int_rel1}
	g(u) = \alpha \Rightarrow \lambda \abs{u} \ge \alpha \Rightarrow \abs{u} \ge \tfrac{\alpha}{\lambda}.
	\end{equation} 
	Also note that for all $\abs{u}\le \lambda$, we have $g(u) = \lambda \abs{u}$ and $\grad h(u) = 0$ which implies $\grad h(u) = 0$ for all $g(u) = \alpha \le \lambda^2$. Hence, using this relation along with \eqref{eq:int_rel1}, we obtain 
	\begin{equation}
	\label{eq:int_rel2}
	(\lambda -\abs{\grad h(u)}) \abs{u} = \lambda \abs{u} \ge \alpha \qquad \text{if } 0 \le \alpha \le \lambda^2.
	\end{equation} We note that $\abs{\grad h(u)} = \lambda$ for all $u \ge \lambda \theta$ and $g(u) =\alpha = \tfrac{\lambda^2(\theta+1)}{2}$ for all $u \ge \lambda\theta$. Hence, 
	\begin{equation}
	\label{eq:int_rel3}
	(\lambda-\abs{\grad h(u)})\abs{u} = 0 \qquad \text{if }\alpha = \tfrac{\lambda^2(\theta+1)}{2}.
	\end{equation} Now we design a lower bound when $\alpha \in (\lambda^2, \tfrac{\lambda^2(\theta+1)}{2})$. For such values of $\alpha$, we have 
	\begin{align*}
	g(u) = &\lambda \abs{u} - \tfrac{(\abs{u}-\lambda)^2}{2(\theta-1)} = \alpha\\
	\Rightarrow &u^2 - 2\lambda\theta \abs{u} + \lambda^2 + 2\alpha(\theta-1)= 0\\
	\Rightarrow& \abs{u} = \lambda\theta - \sqrt{2(\theta-1)\big[ \tfrac{\lambda^2(\theta+1)}{2} - \alpha \big] }\\
	\Rightarrow &\abs{\grad h(u)} = \tfrac{\abs{u} - \lambda}{\theta-1} = \lambda - \sqrt{\tfrac{2}{\theta-1}}\sqrt{\tfrac{\lambda^2(\theta+1)}{2} - \alpha}\\
	\Rightarrow &\lambda - \abs{\grad h(u)} =\sqrt{\tfrac{2}{\theta-1}}\sqrt{\tfrac{\lambda^2(\theta+1)}{2} - \alpha}.
	\end{align*}
	Then, above relation along with \eqref{eq:int_rel1}, we have $(\lambda-\abs{\grad h(u)}) \abs{u} \ge \sqrt{\tfrac{2}{\theta-1}}\tfrac{\alpha}{\lambda}\sqrt{\tfrac{\lambda^2(\theta+1)}{2} - \alpha}$ for all $\alpha \in (\lambda^2, \tfrac{\lambda^2(\theta+1)}{2})$. Using this relation along with \eqref{eq:int_rel2}, \eqref{eq:int_rel3} and noting the definition of function $z(\cdot)$, we obtain a lower bound $(\lambda-\abs{\grad h(u)}) \abs{u} \ge z(\alpha)$ where $\alpha = g(u)$.
	
	Now note that for general high-dimensional $x \in \Rbb^d$, we have $g(x) = \tsum_{i=1}^d g(x_i)= \alpha$. Then $\alpha \in [0, \tfrac{d\lambda^2(\theta+1)}{2}]$. Since each individual $g(x_i) \ge 0$, we can think of $\alpha$ as a budget such that sum of $g(x_i)$ must equal $\alpha$. In order to minimize the lower bound on $(\lambda - \abs{\grad h(x_i)}) \abs{x_i}$, we should exhaust the largest budget from $\tsum_{i=1}^dg(x_i)  = \alpha$ while maintaining the lowest possible value of the lower bound on $(\lambda - \abs{\grad h(x_i)}) \abs{x_i}$. This clearly holds by setting $\abs{x_i}$ such that $g(x_i) =  \tfrac{\lambda^2(\theta+1)}{2}$. This can be clearly observed in the figure below.
	\begin{figure}[H]
		\centering
		\includegraphics[width=0.4\linewidth]{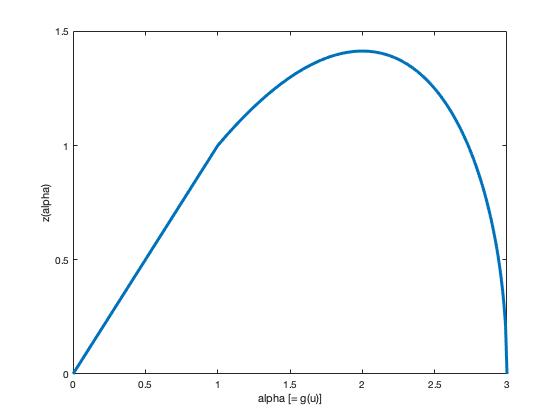}
		\caption{Plot of function $z(\alpha)$ on $y$-axis and $\alpha$ on $x$-axis for $\lambda = 1, \ \theta = 5$. The largest possible value $g(u)$ is $\tfrac{\lambda^2(\theta+1)}{2} = 3$ is achieved for $u \ge \lambda\theta = 5$ and lower bound $z(3) = 0$. Hence, setting $u \ge \lambda \theta$ maximizes the $g(u)$ and minimizes $z(\alpha) = z(g(u))$.}
	\end{figure}
	Hence, if $\alpha \in \Big[\beta\tfrac{\lambda^2(\theta+1)}{2} , (\beta+1)\tfrac{\lambda^2(\theta+1)}{2} \Big)$ for some nonnegative integer $\beta$, then we should set $\beta$ coordinates of $x$ satisfying $\abs{x_i} \ge \lambda\theta$ in order to exhaust the maximum possible budget, $\tfrac{\lambda^2(\theta+1)}{2}$, from $\alpha$ and still keep the value of the lower bound on $(\lambda - \abs{\grad h(u)}) \abs{u}$ as $0$. Hence, noting the definition of $\gamma$, the problem reduces to $\tsum_{i} g(x_i) = \gamma$ where summation is taken over remaining coordinates of $x$ and $\gamma \in \big[0, \tfrac{\lambda^2(\theta+1)}{2} \big)$. 
	
	Lets recall from the analysis in 1-D case that if $g(x_i) = \alpha_i$ then $(\lambda-\abs{\grad h(x_i)}) \abs{x_i} \ge z(\alpha_i)$ so we obtain the lower bound $\tsum_{i} z(\alpha_i)$ while $\alpha_i$'s satisfy the relation $\tsum_{i}\alpha_i = \gamma$. Moreover, $z: [0, \tfrac{\lambda^2(\theta+1)}{2}] \to \Rbb_{\ge 0}$ is a concave function with $z(0) = 0$. Then we show that $z$ is a subadditive function. Using Jensen's inequality, for all $t\in [0,1]$, we have $z(tx + (1-t)y) \ge tz(x) + (1-t)z(y)$. Using $y = 0$ and the fact that $z(0) = 0$, we have $z(tx) \ge t z(x)$ for any $t \in [0,1]$. Now using this relation along with $t = \tfrac{x}{x+y} \in [0,1]$ (for $x, y \ge 0$) we have 
	\begin{align*}
	z(x) &= z(t(x+y)) \ge t z(x+y).\\
	z(y) &= z((1-t)(x+y)) \ge (1-t)z(x+y).
	\end{align*} 
	Adding the two relations, we obtain $z(x) + z(y) \ge z(x+y)$. Hence, $z$ is a subadditive function. Since $\tsum_{i} \alpha_i = \gamma$ then the we have $\tsum_i z(\alpha_i) \ge z(\tsum_{i} \alpha_i) = z(\gamma)$. This bound is indeed achieved when we set one of $\alpha_i = \gamma$ and rest to $0$. Hence, we conclude the proof. 

\subsection{Proof of Theorem \ref{thm:mcp_y_bd_val}}
	As before, we proceed by assuming 1-D case, i.e., $u \in \Rbb$ and $g(u) = \alpha$ and then extend it to general d-dimensional setting. Then, $\alpha \in  [0, \tfrac{\lambda^2\theta}{2}]$. Then, we write function $(\lambda - \abs{\grad h(u)})\abs{u}$ in term of $\alpha$. Note that 
	\begin{align*}
	g(u) &= \lambda\abs{u} - \tfrac{u^2}{2\theta} = \alpha\\
	\Rightarrow \abs{u} &= \theta\lambda \big( 1- \sqrt{1- \tfrac{2\alpha}{\theta\lambda^2}}\big)\\
	\Rightarrow \abs{\grad h(u)} &= \tfrac{\abs{u}}{\theta} = \lambda\big( 1- \sqrt{1- \tfrac{2\alpha}{\theta\lambda^2}} \big)\\
	\Rightarrow \lambda - \abs{\grad h(u)} &= \lambda \sqrt{1- \tfrac{2\alpha}{\theta\lambda^2}}
	\end{align*}
	Moreover, we also have \eqref{eq:int_rel1}. Then, noting the definition of $z(\cdot)$, we obtain that $(\lambda-\abs{\grad h(u)})\abs{u} \ge z(\alpha)$.
	
	For high dimensional $x\in \Rbb^d$, we use similar arguments as in the proof of theorem \ref{thm:dual_bd_value}. In particular, we set $\beta$ coordinates $x$ satisfying $\abs{x_i} \ge \lambda \theta$ which exhausts the maximum possible budget $\tfrac{\lambda^2\theta}{2}$ from $\alpha$ and still keeps the value of the lower bound on $(\lambda - \abs{\grad h(x_i)} )\abs{x_i}$ as $0$. Finally, we reduce the problem to $\tsum_{i} g(x_i) = \tsum_{i} \alpha_i= \gamma$ and lower bound is $\tsum_{i}z(\alpha_i)$. As in the previous case, $z$ is concave function on nonnegative domain with $z(0) = 0$ hence it must be subadditive. So we obtain that $\tsum_{i} z(\alpha_i) \ge z(\tsum_{i}\alpha_i) = z(\gamma)$. Hence, we conclude the proof.

\subsection{Proof of Proposition \ref{prop:SCAD_eta_setting}}
	
	We note that $\eta = \beta \tfrac{\lambda^2(\theta+1)}{2} + \wt{\eta}$, where $\beta$ is the largest nonnegative integer such that $\wt{\eta} \ge 0$. Clearly $\wt{\eta} \in \big[0, \tfrac{\lambda^2(\theta+1)}{2}\big)$. Now, we divide our analysis in two cases:\\
	{\bf Case 1:} Suppose $\wt{\eta} \le \lambda^2$. Then we define $
	\eta_{0}$ for Algorithm \ref{alg:main} as $\eta_{0} = \beta \tfrac{\lambda^2(\theta+1)}{2} + \tfrac{\wt{\eta}}{2}$.\\
	Now, if $g(x^{k-1}) \le \beta\tfrac{\lambda^2(\theta+1)}{2}$ then we have that $\eta_{k-1}-  g(x^{k-1}) \ge \eta_{0} - g(x^{k-1}) \ge \tfrac{\wt{\eta}}{2}$. In this case, we obtain that denominator of \eqref{eq:bound_y_bar} is at least $\tfrac{\wt{\eta}}{2}$.\\ 
	In other case, suppose that $g(x^{k-1}) > \beta\tfrac{\lambda^2(\theta+1)}{2}$. We also note that $g(x^{k-1}) \le g_{k-1}(x^{k-1}) \le \eta_{k-1} \le \eta$. Hence, we obtain $g(x^{k-1}) \le \eta = \beta\tfrac{\lambda^2(\theta+1)}{2} + \wt{\eta}$. This implies $\wt{g}(x^{k-1}):=g(x^{k-1})-\beta\tfrac{\lambda^2(\theta+1)}{2} \in [0, \lambda^2]$. Then, using Theorem \ref{thm:dual_bd_value}, we obtain that $\tsum_{i=1}^d (\lambda-\abs{\grad h(x^{k-1}_i)}) \abs{x^{k-1}_i} \ge z(\wt{g}(x^{k-1})) = \wt{g}(x^{k-1})$. Using this relation, we obtain that $\eta_{k-1} -g(x^{k-1}) + \tsum_{i=1}^d (\lambda-\abs{\grad h(x^{k-1}_i)}) \abs{x^{k-1}_i} \ge \eta_{k-1} -g(x^{k-1})  + \wt{g}(x^{k-1}) = \eta_{k-1} - \beta\tfrac{\lambda^2(\theta+1)}{2} = \wt{\eta}_{k-1} \ge \tfrac{\wt{\eta}}{2}$.\\
	So, when $\wt{\eta} \le \lambda^2$, we obtain that the denominator in \eqref{eq:bound_y_bar} is at least $\eta_{k}-\eta_{k-1}+ \tfrac{z(\wt{\eta})}{2} = \delta_{k} + \tfrac{z(\wt{\eta})}{2} \ge  \tfrac{z(\wt{\eta})}{2}$. 
	
	{\bf Case 2:} Now, we look at the second case where $\wt{\eta} > \lambda^2$. In this case, we define $\eta_{0} = \beta \tfrac{\lambda^2(\theta+1)}{2} + \min\{\lambda^2, z(\wt{\eta})\}$. Then, we again note that $g(x^{k-1}) \le \beta\tfrac{\lambda^2(\theta+1)}{2}$ implies $\eta_{k-1} - g(x^{k-1}) \ge \wt{\eta}_{k-1} \ge \wt{\eta}_{0}$.\\
	In other case, we assume that $g(x^{k-1}) \in [\beta \tfrac{\lambda^2(\theta+1)}{2}, \beta \tfrac{\lambda^2(\theta+1)}{2} + \lambda^2]$, then again using Theorem \ref{thm:dual_bd_value}, we obtain $\tsum_{i=1}^d (\lambda-\abs{\grad h(x^{k-1}_i)}) \abs{x^{k-1}_i} \ge z(\wt{g}(x^{k-1})) = \wt{g}(x^{k-1})$. This implies $\eta_{k-1} - g(x^{k-1}) + \tsum_{i=1}^d (\lambda-\abs{\grad h(x^{k-1}_i)}) \abs{x^{k-1}_i} \ge \eta_{k-1} - \beta \tfrac{\lambda^2(\theta+1)}{2} = \wt{\eta}_{k-1} \ge \wt{\eta}_0$.\\
	Finally, $g(x^{k-1}) > \beta\tfrac{\lambda^2(\theta+1)}{2} + \lambda^2$ then $\wt{g}(x^{k-1}) \in (\lambda^2, \wt{\eta})$ then due to concavity of $z$, we obtain that $z(\wt{g}(x^{k-1})) \ge \min\{\lambda^2, z(\wt{\eta})\} = \wt{\eta}_{0}$. 
	
	Hence, combining the bounds in both cases, we obtain that denominator in \eqref{eq:bound_y_bar} is always bounded below by $\min\{\lambda^2, \tfrac{z(\eta)}{2}\}$.

\section{Proof of Theorem \ref{thm:complexity_main}}
As in the previous case, we show an important recursive property of iterates.
We first state the theorem again:
\begin{thm}
	Suppose Assumption \ref{assu:feasibility-level}, \ref{assu:bound-y-1} hold such that $\delta_{k} = \tfrac{\eta-\eta_{0}}{k(k+1)}$ for all $k \ge 1$. Let $\pi_k$ denote the randomness of $x^1, \dots, x^{k-1}$. Suppose for $k$-th subproblem \eqref{subprob}, the solution $x^k$ satisfies 
	\begin{align*}
	\Ebb\bracket{\psi_{k}(x^{k})-\psi_{k}(\bar{x}^{k})|\pi_k} &\le\tfrac{\rho}{2}\norm{x^{k-1}-\bar{x}^{k}}^{2}+\zeta_k,\\
	g_k(x^k) &\le \eta_{k}
	\end{align*}
	where $\rho$ lies in the interval $[0,\gamma-\mu]$ and $\{\zeta_{k}\}$ is a sequence of nonnegative numbers. If $\hat{k}$ is chosen uniformly randomly from $\left\lfloor \tfrac{K+1}{2} \right\rfloor$ to K then corresponding to $x^{\hat{k}}$, there exists pair $(\wb{x}^{\hat{k}}, \wb{y}^{\hat{k}})$ satisfying
	\begin{align*}
	&\Ebb_{\hat{k}}\big[ \dis\big(\partial_{x}\Lcal(\bar{x}^{\hat{k}},\bar{y}^{\hat{k}}),0\big)^{2}\big] \le \tfrac{8(\gamma^2 + B^2L_h^2)}{K(\gamma-\mu-\rho)}\big(\tfrac{\gamma-\mu + \rho}{\gamma-\mu}\Delta^0 + 2Z_1\big),\\
	&\Ebb_{\hat{k}}\big[\bar{y}^{\hat{k}}\big|g(\bar{x}^{\hat{k}})-\eta\big| \big] \le \tfrac{2BL_h}{K(\gamma-\mu-\rho)}\big(\tfrac{\gamma-\mu + \rho}{\gamma-\mu}\Delta^0 + 2Z_1\big) + \tfrac{2B(\eta-\eta_{0})}{K},\\
	&\Ebb_{\hat{k}}\norm{x^{\hat{k}}-\bar{x}^{\hat{k}}}^{2}\le \tfrac{4\rho(\gamma-\mu+\rho)}{K(\gamma-\mu)^2(\gamma-\mu-\rho)}\Delta^0 + \tfrac{8Z_1}{K(\gamma-\mu-\rho)},
	\end{align*}
	where, $\Delta^0 := \psi(x^0) - \psi(x^*)$ and $Z_1 := \tsum_{k=1}^K \zeta_{k}$.
\end{thm}
We first prove the following important relationship on the sum of squares of distances of the iterates.
\begin{prop}
\label{prop:composite-2}Let requirements of Theorem \ref{thm:complexity_main} hold.
Then for any $s\ge2$, we have
\begin{align}
\Ebb\bracket{\tsum_{k=s}^{K}\norm{x^{k-1}-\bar{x}^{k}}^{2}|\pi_{s-1}} & \le\tfrac{2(A_s+Z_s)}{\gamma-\mu-\rho},\label{eq:sum-square-3}\\
\Ebb\bracket{\tsum_{k=s}^{K}\norm{x^{k}-\bar{x}^{k}}^{2}|\pi_{s-1}} & \le\tfrac{2\rho A_s}{(\gamma-\mu)(\gamma-\mu-\rho)} + \tfrac{2Z_s}{\gamma-\mu-\rho}\label{eq:sum-square-4}
\end{align}
where $A_s=\tfrac{\gamma-\mu+\rho}{\gamma-\mu}\left[\psi(x^{s-2})-\psi(x^{*})\right]$ and $Z_s = \tsum_{k=s-1}^K\zeta_{k}$.
\end{prop}
\begin{proof}
	Note that since for all $k \ge 1$ we have feasibility of $x^k$ for $k$-th subproblem (due to \eqref{eq:subprob_feas}), then in view of Proposition \ref{prop:KKT-subseq}, we have that $x^{k-1}$ is strictly feasible for the $k$-th subproblem. Consequently, using strong convexity of $\psi_{k}$ and optimality of $\wb{x}^k$, we have $\tfrac{\gamma-\mu}{2}\norm{x^{k-1} - \wb{x}^k}^2 \le \psi_{k}(x^{k-1}) - \psi_{k}(\wb{x}^k)$.
	Therefore, taking expectation conditioned on $\pi_{k-1}$ ob both sides of the above relation, we obtain
	\begin{align*}
\tfrac{\gamma-\mu}{2}&\Ebb\bracket{\norm{x^{k-1}-\bar{x}^{k}}^{2}|\pi_{k-1}} \le \Ebb\bracket{ \psi_{k}(x^{k-1})-\psi_{k}(\bar{x}^{k})|\pi_{k-1}}\\
 & \le \Ebb\bracket{\psi_{k-1}(x^{k-1})-\psi_{k}(\bar{x}^{k})|\pi_{k-1}}\\
 &\le\psi_{k-1}(\bar{x}^{k-1})-\Ebb\bracket{\psi_{k}(\bar{x}^{k})|\pi_{k-1}}+\tfrac{\rho}{2}\norm{x^{k-2}-\bar{x}^{k-1}}^{2}+\zeta_{k-1}
\end{align*}
where second inequality follows from $\psi_{k}(x^{k-1})=\psi(x^{k-1})\le\psi_{k-1}(x^{k-1})$
and third inequality follows from \eqref{inexactness}.
Placing the definition of $\psi_{k}(\cdot)$ in above relation, we
have 
\begin{align*}
\tfrac{2\gamma-\mu}{2}\Ebb\bracket{\norm{x^{k-1}-\bar{x}^{k}}^{2}|\pi_{k-1}} &\le\psi(\bar{x}^{k-1})-\Ebb\bracket{\psi(\bar{x}^{k})|\pi_{k-1}}+\tfrac{\gamma+\rho}{2}\norm{x^{k-2}-\bar{x}^{k-1}}^{2}+\zeta_{k-1}.
\end{align*}
Summing up over $k=s,s+1,\dots,K$ and taking expectation conditioned on $\pi_{s-1}$, we have 
\begin{align*}
\tfrac{2\gamma-\mu}{2}\tsum_{k=s}^{K}\Ebb\big[\norm{x^{k-1}-\bar{x}^{k}}^{2} | \psi_{s-1}\big] & \le\psi(\bar{x}^{s-1})-\Ebb\psi(\bar{x}^{K})\\
&\quad+\tfrac{\gamma+\rho}{2}\tsum_{k=s}^{K}\Ebb\big[\norm{x^{k-2}-\bar{x}^{k-1}}^{2} | \pi_{s-1}\big]+\tsum_{k=s}^K\zeta_{k-1}.
\end{align*}
It then follows that 
\begin{align*}
\tfrac{\gamma-\mu-\rho}{2}\Ebb\big[\tsum_{k=s}^{K}\norm{x^{k-1}-\bar{x}^{k}}^{2} | \pi_{s-1}\big] & \le\psi(\bar{x}^{s-1})-\Ebb\,\psi(\bar{x}^{K})+\tfrac{\gamma+\rho}{2}\norm{x^{s-2}-\bar{x}^{s-1}}^{2}+\tsum_{k=s}^K\zeta_{k-1}\\
 & \le\psi_{s-1}(\bar{x}^{s-1})-\Ebb\psi(\bar{x}^{K})\\
 &\qquad+\tfrac{\rho}{\gamma-\mu}\left[\psi_{s-1}(x^{s-2})-\psi_{s-1}(\bar{x}^{s-1})\right]+\tsum_{k=s}^K\zeta_{k-1}\\
 & \le\psi(x^{s-2})-\Ebb\psi(\bar{x}^{K})\\
 &\qquad+\tfrac{\rho}{\gamma-\mu}\left[\psi(x^{s-2})-\psi_{s-1}(\bar{x}^{s-1})\right]+\tsum_{k=s}^K\zeta_{k-1}\\
 & \le \tfrac{\gamma-\mu+\rho}{\gamma-\mu}\left[\psi(x^{s-2})-\psi(x^{*})\right]+\tsum_{k=s}^K\zeta_{k-1},
\end{align*}
where the third and the last inequality follow from the property $$\psi(x^{k-1})=\psi_k(x^{k-1})\ge\psi_{k}(\bar{x}^{k})\ge\psi(\bar{x}^{k})\ge\psi(x^{*}).$$ Note that solution $x^k$ 
is feasible for the $k$-th subproblem and hence, in view of Proposition \ref{prop:KKT-subseq}, we have that $g(\wb{x}^k) \le g_k(\wb{x}^k) \le \eta_{k} < \eta$ and hence $\wb{x}^k$ is feasible solution for the main problem implying $\psi(\wb{x}^k) \ge \psi(x^*)$ in the above relation.
Then (\ref{eq:sum-square-3}) immediately follows. 

Now we prove that (\ref{eq:sum-square-4}) holds. Note that
\begin{align*}
\Ebb\left[\norm{x^{k}-\bar{x}^{k}}^{2} | \pi_k\right]&\le\tfrac{2}{\gamma-\mu}\Ebb \left[\psi_{k}(x^{k})-\psi_{k}(\bar{x}^{k}) | \pi_k\right] \le\tfrac{2}{\gamma-\mu}\left[\tfrac{\rho}{2}\norm{x^{k-1}-\bar{x}^{k}}^{2} + \zeta_{k}\right],
\end{align*}
where the first inequality follows due to the strong convexity $\psi_{k}$ as well as the optimality of $\wb{x}^k$ and the second inequality follows due to \eqref{inexactness}. Now summing the above relation from $k = s$ to $K$ and taking expectation conditioned on $\psi_{s-1}$, we obtain
\begin{align*}
\Ebb\left[\tsum_{k=s}^{K}\norm{x^{k}-\bar{x}^{k}}^{2} | \pi_{s-1}\right]&\le \tfrac{\rho}{\gamma-\mu}\Ebb\left[\tsum_{k=s}^K\norm{x^{k-1}-\bar{x}^{k}}^{2} | \pi_{s-1}\right] + \tfrac{2}{\gamma-\mu} \tsum_{k=s}^K \zeta_{k}\\
& \le\tfrac{2\rho A_s}{(\gamma-\mu)(\gamma-\mu-\rho)} + \tfrac{2Z_s}{\gamma-\mu-\rho},
\end{align*}
where the last inequality follows from \eqref{eq:sum-square-3} and the definition of $Z_s$. Hence, we conclude the proof.
\end{proof}
Now we present the unified convergence of proximal point as stated in Theorem \ref{thm:complexity_main}.
\begin{proof}[Proof of Theorem \ref{thm:complexity_main}]
	Due to the KKT condition for the subproblem \eqref{subprob}, we have
	\begin{equation}\label{eq:KKT_subprob1}
	\begin{split}
	0 & \in\partial\psi(\bar{x}^{k})+\gamma\left(\bar{x}^{k}-x^{k-1}\right)+\bar{y}^{k}\left(\partial\norm{\bar{x}^{k}}_{1}-\nabla h(x^{k-1})\right)\\
	0 & =\bar{y}^{k}\left(\lambda\norm{\bar{x}^{k}}_{1}-h(x^{k-1})-\inprod{\nabla h(x^{k-1})}{\bar{x}^{k}-x^{k-1}}-\eta_{k}\right)
	\end{split}
	\end{equation}
	Using triangle inequality along with first relation in the above equation, we have $\dis\left(\partial_{x}\Lcal(\bar{x}^{k},\bar{y}^{k}),0\right)\le\gamma\norm{\bar{x}^{k}-x^{k-1}}+\bar{y}^{k}\norm{\nabla h(x^{k-1})-\nabla h(\bar{x}^{k})}$.
	Therefore, noting the bound on $\wb{y}^k$ from Assumption \ref{assu:bound-y-1}, we have
	\begin{align*}
	\dis\left(\partial_{x}\Lcal(\bar{x}^{k},\bar{y}^{k}),0\right)^{2}  &\le 2\gamma^{2}\norm{\bar{x}^{k}-x^{k-1}}^{2}+2B^{2}\norm{\nabla h(x^{k-1})-h(\bar{x}^{k})}^{2}\\
	& \le 2\left(\gamma^{2}+B^{2}L_{h}^{2}\right)\norm{\bar{x}^{k}-x^{k-1}}^{2},
	\end{align*}
	where the second inequality uses Lipschitz smoothness of $h(x)$. Summing the above relation from $k = s, \dots, K$ and the taking expectation conditioned on $\pi_{s-1}$ on both sides, we obtain
	\begin{align}
	\Ebb\left[\tsum_{k=s}^{K}\dis\left(\partial_{x}\Lcal(\bar{x}^{k},\bar{y}^{k}),0\right)^{2}|\pi_{s-1}\right]& \le 2(\gamma^2+B^2L_h^2) \Ebb\left[\tsum_{k=s}^{K} \norm{x^{k-1}-\bar{x}^{k}}^{2} |\pi_{s-1}\right]\nonumber\\
	& \le \tfrac{4\left(\gamma^{2}+B^{2}L_{h}^{2}\right)}{\gamma-\mu-\rho}(A_s+ Z_s),\label{eq:int_rel4}
	\end{align}
	For the complementary slackness part of the KKT condition, first notice that $\eta_{k}=\eta_{0}+\tsum_{t=1}^{k}\delta_{t}=\eta_{0}+\tsum_{t=1}^{k}\tfrac{\eta-\eta_{0}}{t(t+1)}=\tfrac{k}{k+1}\eta+\tfrac{1}{k+1}\eta_{0}$.
	Therefore, 
	\[
	\tsum_{k=s}^{K}\left(\eta-\eta_{k}\right)=\tsum_{k=s}^{K}\tfrac{\eta-\eta_{0}}{k+1}\le\tfrac{K+1-s}{s+1}(\eta-\eta_{0}).
	\]
	To prove the error of complementary slackness condition, observe that
	\begin{align*}
	\bar{y}^{k}\left|\lambda\norm{\bar{x}^{k}}_{1}-h(\bar{x}^{k})-\eta\right| & \le\bar{y}^{k}\left|\lambda\norm{\bar{x}^{k}}_{1}-h(x^{k-1})-\inprod{\nabla h(x^{k-1})}{\bar{x}^{k}-x^{k-1}}-\eta_k\right|\\
	& \quad+\bar{y}^{k}\left|h(x^{k-1})+\inprod{\nabla h(x^{k-1})}{\bar{x}^{k}-x^{k-1}}-h(\bar{x}^{k})\right|+\bar{y}^{k}\left(\eta-\eta_{k}\right)\\
	& \le\tfrac{BL_{h}}{2}\norm{\bar{x}^{k}-x^{k-1}}^{2}+B\left(\eta-\eta_{k}\right),
	\end{align*}
	where second inequality follows due to second relation in \eqref{eq:KKT_subprob1} and bound on $\wb{y}^k$ from Assumption \ref{assu:bound-y-1}. Summing the above relation from $k = s, \dots, K$ and taking expectation conditioned on $\pi_{s-1}$ on both sides, we obtain 
	\begin{align}
	\Ebb\left[\tsum_{k=s}^{K}\bar{y}^{k}\left|g(\bar{x}^k)-\eta\right||\pi_{s-1}\right] & \le\tsum_{k=s}^{K}\Ebb \left[\tfrac{BL_{h}}{2}\norm{\bar{x}^{k}-x^{k-1}}^{2}+B\left(\eta-\eta_{k}\right)|\pi_{s-1}\right]\nonumber\\
	& \le\tfrac{BL_{h}}{2}\Ebb \left[\tsum_{k=s}^{K}\norm{\bar{x}^{k}-x^{k-1}}^{2} | \psi_{s-1}\right]+B\tsum_{k=s}^{K}\left(\eta-\eta_{k}\right)\nonumber\\
	& \le\tfrac{BL_{h}}{\gamma-\mu-\rho}(A_s+Z_s)+\tfrac{\left(K+1-s\right)B(\eta-\eta_0)}{s+1}.\label{eq:int_rel6}
	\end{align}
	Now note that $A_s = \tfrac{\gamma-\mu+\rho}{\gamma-\mu} [\psi(x^{s-2})- \psi(x^*)]$ is a random variable due to randomness of $x^{s-2}$.  Now we bound expectation of $\psi(x^{s-2})$.
	In view of \eqref{inexactness}, we have
	\begin{align*}
	\Ebb[\psi_k(x^k) | \pi_k] &\le \psi_k(\wb{x}^k) + \tfrac{\rho}{2}\norm{x^{k-1}-\wb{x}^k}^2 + \zeta_{k}\\
	&\le \psi_k(x^{k-1}) - \tfrac{\gamma-\mu-\rho}{2}\norm{x^{k-1}-\wb{x}^k} + \zeta_{k}
	\end{align*}
	Since, $\gamma-\mu-\rho \ge 0$ and noting that $\psi_{k}(x^{k-1}) = \psi(x^{k-1})$, $\psi_{k}(x^k) \ge \psi(x^k)$, we have
	\[\Ebb[\psi(x^k)|\pi_k] \le \psi(x^{k-1}) + \zeta_{k}. \]
	Taking expectation on both sides of the above relation and then summing from $k = 1$ to $s-2$, we get
	\[ \Ebb[\psi(x^{s-2})] \le \psi(x^0) + \tsum_{k=1}^{s-2}\zeta_{k}.\]
	Using the above relation, we obtain 
	\begin{equation}\label{eq:int_rel5}
	\Ebb[A_s] \le \tfrac{\gamma-\mu + \rho}{\gamma-\mu}\Delta^0 + 2\tsum_{k=1}^{s-2}\zeta_{k},
	\end{equation}
	where $\Delta^0 = \psi(x^0)-\psi(x^*)$.
	Note that here we used the fact $\tfrac{\gamma-\mu+\rho}{\gamma-\mu} \le 2$. 
	Now taking expectation on both sides of \eqref{eq:int_rel4} and using bound on $\Ebb[A_s]$ in \eqref{eq:int_rel5}, we obtain
	\begin{align*}
	\Ebb\left[\tsum_{k=s}^{K}\dis\left(\partial_{x}\Lcal(\bar{x}^{k},\bar{y}^{k}),0\right)^{2}|\pi_{s-1}\right] &\le \tfrac{4(\gamma^2 + B^2L_h^2)}{\gamma-\mu-\rho}\big(\tfrac{\gamma-\mu + \rho}{\gamma-\mu}\Delta^0  + 2\tsum_{k=1}^{s-2}\zeta_k + \tsum_{k=s-1}^K\zeta_{k}\big)\\
	&\le\tfrac{4(\gamma^2 + B^2L_h^2)}{\gamma-\mu-\rho}\big(\tfrac{\gamma-\mu + \rho}{\gamma-\mu}\Delta^0 + 2Z_1\big).
	\end{align*}
	Similarly, taking expectation on both sides of \eqref{eq:int_rel6} and using \eqref{eq:int_rel5}, we obtain
	\begin{align*}
	\Ebb\left[\tsum_{k=s}^{K}\bar{y}^{k}\left|g(\bar{x}^k)-\eta\right||\pi_{s-1}\right] &\le \tfrac{BL_h}{\gamma-\mu-\rho}\big(\tfrac{\gamma-\mu + \rho}{\gamma-\mu}\Delta^0 + 2Z_1\big) + \tfrac{K+1-s}{s+1}B(\eta-\eta_{0}).
	\end{align*}
	Taking expectation on both sides of \eqref{eq:sum-square-4} and using \eqref{eq:int_rel5}, we obtain 
	\begin{align*}
	\Ebb\bracket{\tsum_{k=s}^{K}\norm{x^{k}-\bar{x}^{k}}^{2} } &\le \tfrac{2\rho}{(\gamma-\mu)(\gamma-\mu-\rho)}\big( \tfrac{\gamma-\mu+\rho}{\gamma-\mu}\Delta^0 + 2\tsum_{k=1}^{s-2}\zeta_{k} \big) + \tfrac{2Z_s}{\gamma-\mu-\rho}\\
	&\le \tfrac{2\rho(\gamma-\mu+\rho)}{(\gamma-\mu)^2(\gamma-\mu-\rho)}\Delta^0 + \tfrac{4Z_1}{\gamma-\mu-\rho}.
	\end{align*}
	Finally, setting $s = \left\lfloor\tfrac{K+1}{2}\right\rfloor$, we have $\tfrac{K}{2} \le s \le \tfrac{K+1}{2}$. Therefore, we have
	\begin{align*}
	\Ebb_{\hat{k}}\left[ \dis\left(\partial_{x}\Lcal(\bar{x}^{\hat{k}},\bar{y}^{\hat{k}}),0\right)^{2}\right] &\le \tfrac{8(\gamma^2 + B^2L_h^2)}{K(\gamma-\mu-\rho)}\big(\tfrac{\gamma-\mu + \rho}{\gamma-\mu}\Delta^0 + 2Z_1\big),\\
	\Ebb_{\hat{k}}\left[\bar{y}^{\hat{k}}\left|g(\bar{x}^{\hat{k}})-\eta\right|\right] & \le \tfrac{2BL_h}{K(\gamma-\mu-\rho)}\big(\tfrac{\gamma-\mu + \rho}{\gamma-\mu}\Delta^0 + 2Z_1\big) + \tfrac{2B(\eta-\eta_{0})}{K},	
	\end{align*}
	and
	\begin{align*}
	&\Ebb_{\hat{k}}\norm{x^{\hat{k}}-\bar{x}^{\hat{k}}}^{2}\le \tfrac{4\rho(\gamma-\mu+\rho)}{K(\gamma-\mu)^2(\gamma-\mu-\rho)}\Delta^0 + \tfrac{8Z_1}{K(\gamma-\mu-\rho)}.
	\end{align*}

Hence, we conclude the proof.
\end{proof}

\subsection{Proof of Corollary \ref{cor:unified_complexity_main}}
Since $T_k \ge 2\sqrt{\tfrac{L}{\mu}+3}$, we have that $\tfrac{2(L+\gamma)}{T_k^2} = \tfrac{2(L+3\mu)}{T_k^2} \le \tfrac{\mu}{2} = \tfrac{\rho}{2}$. Moreover, we see that $\rho = \mu \le \gamma- \mu = 2\mu$. Finally, since $T_k \ge K(M+\sigma)$ so we have $\zeta_{k} \le \tfrac{4}{\mu K}$ implying that $Z_1 = \tsum_{k=1}^K\zeta_{k} \le \tfrac{4}{\mu}$. Then, applying Theorem \ref{thm:complexity_main}, we obtain that $x^{\hat{k}}$ is an $(\vep_1, \vep_2)$-KKT solution of the problem \eqref{noncvx-constraint}.
\subsection{Convergence for the (stochastic) convex case}
We have the following Corollary of Theorem \ref{thm:complexity_main} for the case in which objective $\psi$ is convex, i.e. $\mu = 0$.
\begin{corollary}\label{cor:unified_complexity_convex}
	Let $\psi$ be convex function such that it satisfies \eqref{eq:low_curv_psi} with $\mu  = 0$. Set $\gamma = \beta L$ where $\beta \in [0, 1)$ be a small constant and run AC-SA for $T_k = \max\{2\sqrt{\tfrac{2(1+\beta)}{\beta}}, K(M+\sigma)\}$ iterations where $K$ is the total number of iterations of Algorithm \ref{alg:main}. Then, $x^{\hat{k}}$ is an $(\vep_1, \vep_2)$-KKT point of the problem \eqref{noncvx-constraint} where
	\[\vep_1 = \big(\tfrac{3\Delta^0}{2K}+ \tfrac{16(M+\sigma)}{\beta KL}\big)\max\{\tfrac{16(\beta^2L^2 + B^2L_h^2)}{\beta L}, \tfrac{4BL_h}{\beta L}\} + \tfrac{2B(\eta-\eta_{0})}{K} ,\]
	\[\vep_2 = \tfrac{3\Delta^0}{2\beta LK} + \tfrac{128(M+\sigma)}{\beta L^2K}.\]
\end{corollary}
\begin{proof}
	Since $T_k \ge 2\sqrt{\tfrac{2(1+\beta)}{\beta}}$, we have $\tfrac{2(L+\gamma)}{T_k^2} = \tfrac{2(1+\beta)L}{T_k^2} \le \tfrac{\beta L}{4} = \tfrac{\rho}{2}$. Moreover, note that $\rho = \tfrac{\beta L}{2} \le \gamma = \beta L$. Finally, since $T_k \ge K(M+\sigma)$ so we have $\zeta_{k} = \tfrac{8(M^2+\sigma^2)}{\gamma T_k} \le \tfrac{8(M+\sigma)}{\beta L K}$. Hence, $Z_1 = \tsum_{k=1}^K\zeta_{k} \le \tfrac{8(M+\sigma)}{\beta L}$. Then, applying Theorem \ref{thm:complexity_main}, we obtain that $x^{\hat{k}}$ is an $(\vep_1, \vep_2)$-KKT solution of problem \eqref{noncvx-constraint}.
\end{proof}
\paragraph*{Finite-sum problem}
A special case of objective takes the finite-sum form $f(x)=\tfrac{1}{n}\tsum_{i=1}^n \wtil{f}_i(x)$ thereby leading to the following subproblem
\[
\min_x \wtil{\psi}(x)=\tfrac{1}{n}\tsum_{i=1}^n \wtil{f}_i(x) + \wtil{\omega}(x)
\]
It is known that finite-sum problem can be efficiently solved by using variance reduction or randomized incremental gradient method \cite{xiao2014proximal,lan2019a}. The complexity of  LCPP on finite-sum problem can be further improved if we apply variance reduction technique for solving the subproblem.
We comment on the complexity result in brief.
In the finite-sum setting, the Nesterov's accelerated gradient-based LCPP requires  $T_k=\wtil{\Ocal}(n\sqrt{\tfrac{L+2\mu}{\mu}})$ and $T_k=\wtil{\Ocal}(n\beta^{-1/2})$ number of stochastic gradient computations to solve each LCPP subproblem. Even though this number is a constant in terms of dependence on $K$, number of terms ($n$) in the finite sum can be large. In comparison to these standard methods, the complexity of SVRG (stochastic variance reduced gradient) based LCPP method can be improved to $T_k=\wtil{\Ocal}(n+{\tfrac{L+\mu}{\mu}})$ for the case when $\psi$ is nonconvex satisfying \eqref{eq:low_curv_psi} with $\mu > 0$, and to $T_k=\wtil{\Ocal}(n+\beta^{-1})$ for convex problem where $\mu = 0$. 
\section{Proof for the projection algorithm for problem \eqref{eq:prox}}\label{apx:prox_oracle}
We formulate the update as the following problem
\begin{equation}\label{projection1}
\min_{x \in \Rbb^d }\, \tfrac{1}{2}\norm{x-v}^2 \ {\text{s.t.}}\ \norm{x}_1+\inprod{u}{x}\le \tau.
\end{equation}
Since the objective is strongly convex, problem \eqref{projection1} has a unique global optimal solution.
Moreover,  the problem is strictly feasible because of the strict feasibility guarantee (\ref{prop:KKT-subseq}) in the context of problem \eqref{subprob}. 
Therefore, KKT condition guarantees that there exists $y\ge 0$ such that
\begin{align}
0 & \in x-v+yu+y\partial\norm x_{1},\label{eq:grad-kkt}\\
0 & =y\left(\inprod ux+\norm x_{1}-\tau\right).\label{eq:cs-kkt}
\end{align}

The algorithm proceeds as follows.
First, we check whether $v$ is feasible, if it is the case, then
$x=v$ is the optimal solution. 
Otherwise, the constraint in \eqref{projection1} is active. Next, we explore the optimality condition \eqref{eq:grad-kkt}.
Given the optimal Lagrangian multiplier $y\ge0$, for the $i$-th coordinate of the optimal $x$, one of the following three situations will 
occur:
\begin{enumerate}
	\item $x_{i}>0$ and $x_{i}=v_{i}-(u_{i}+1)y$.
	\item $x_{i}<0$ and $x_{i}=v_{i}-(u_{i}-1)y$.
	\item $x_{i}=0$ and $(u_i-1)y\le v_{i}\le (u_{i}+1)y$. 
\end{enumerate}
For simplicity, let us denote  $[a]_+=\max\{a,0\}$ and $[a,b]_+=\max\{a, b, 0\}$. 
Based on the discussion above, we can express $x$ as a piecewise linear function of $y$. 
\[
x_{i}(y)=\left[v_{i}-(u_{i}+1)y\right]_{+}-\left[(u_{i}-1)y-v_{i}\right]_{+}.
\]
Let us denote $\ell(y) =\inprod u{x(y)}+\norm{x(y)}_{1}$.
We can deduce that 
\begin{align*}
\ell(y)  &=\tsum_{i=1}^d u_i x_i(y)+\tsum_{i=1}^d \max\{x_i(y), -x_i(y)\}\\	
& = \tsum_{i=1}^d u_i \left[v_{i}-(u_{i}+1)y\right]_{+} - \tsum_{i=1}^d u_i \left[(u_{i}-1)y-v_{i}\right]_{+} \\
& \quad + 2\tsum_{i=1}^d [v_{i}-(u_{i}+1)y, (u_{i}-1)y-v_{i}]_+ \\
& \quad - \tsum_{i=1}^d  \left[v_{i}-(u_{i}+1)y\right]_{+} -  \tsum_{i=1}^d \left[(u_{i}-1)y-v_{i}\right]_{+}\\
& = \tsum_{i=1}^d (u_i-1) \left[v_{i}-(u_{i}+1)y\right]_{+} \\
&\quad- \tsum_{i=1}^d (u_i+1) \left[(u_{i}-1)y-v_{i}\right]_{+} \\
& \quad +2\tsum_{i=1}^d [v_{i}-(u_{i}+1)y, (u_{i}-1)y-v_{i}]_+ 
\end{align*}
Above,  the second equality uses the identity: $\max\{p-q, q-p\}=2\max\{p, q\}- p-q$ for any $p,q\in\Rbb$.
It can be readily seen that $\ell(y)$ is a piecewise linear function with at most $3d$ breaking points.
We can sort these  points in $\Ocal(d\log d)$ and then apply a line-search to find the root of $\ell(\cdot)=\tau$ in $\Ocal(d)$ time.

\section{Supermartingale convergence theorem}
In below, we state a version of supermartingale convergence theorem
developed by \cite{robbins1985a}.
\begin{thm}
\label{thm:supermartingale}Let $(\Omega,F,P)$ be a probability space
and $\Fcal_{0}\subseteq\Fcal_{1}\subseteq...\subseteq\Fcal_{k}\subseteq$
be some sub-$\sigma$-algebra of $F$. Let $b_{k}$, $c_{k}$ be nonnegative
$\Fcal_{k}$-measurable random variables such that 
\[
\Ebb\left[b_{k+1}\mid\Fcal_{k}\right]\le b_{k}+\xi_{k}-c_{k},
\]
where $\{\xi_{k}\}_{0\le k<\infty}$ is a non-negative and summable:
$\tsum_{k=0}^{\infty}\xi_{k}<+\infty$. Then we have 
\[
\lim_{k\raw\infty}b_{k}\text{ exists, and }\tsum_{k=1}^{\infty}c_{k}<+\infty,\quad a.s.
\]
\end{thm}

\newpage
\section{Additional experiments}\label{apx:added_exp}
This section describes additional experiments for investigating the empirical performance of LCPP. We run all the algorithms on a cluster node with Intel Xeon Gold 2.6G CPU and 128G RAM.
\subsubsection*{Solving the subproblems}
We compare the performance
of different instances of LCPP for which the subproblems are solved by
a variety of convex algorithms. Specifically, we consider LCPP-SVRG, LCPP-SGD,
LCPP-NAG and LCPP-BB in which the subproblems are solved by proximal
stochastic variance reduced gradient descent (SVRG \cite{xiao2014proximal}), proximal stochastic gradient descent (SGD), Nesterov's accelerated
gradient (NAG\cite{RN179}) and spectral gradient (Barzilai-Borwein stepsize)  respectively.
We adopt the spectral gradient descent with non-monotone line search from \cite{gong2013a} due to its superior  performance in the reported experiments.

\begin{figure}[h]
	\includegraphics[scale=0.35]{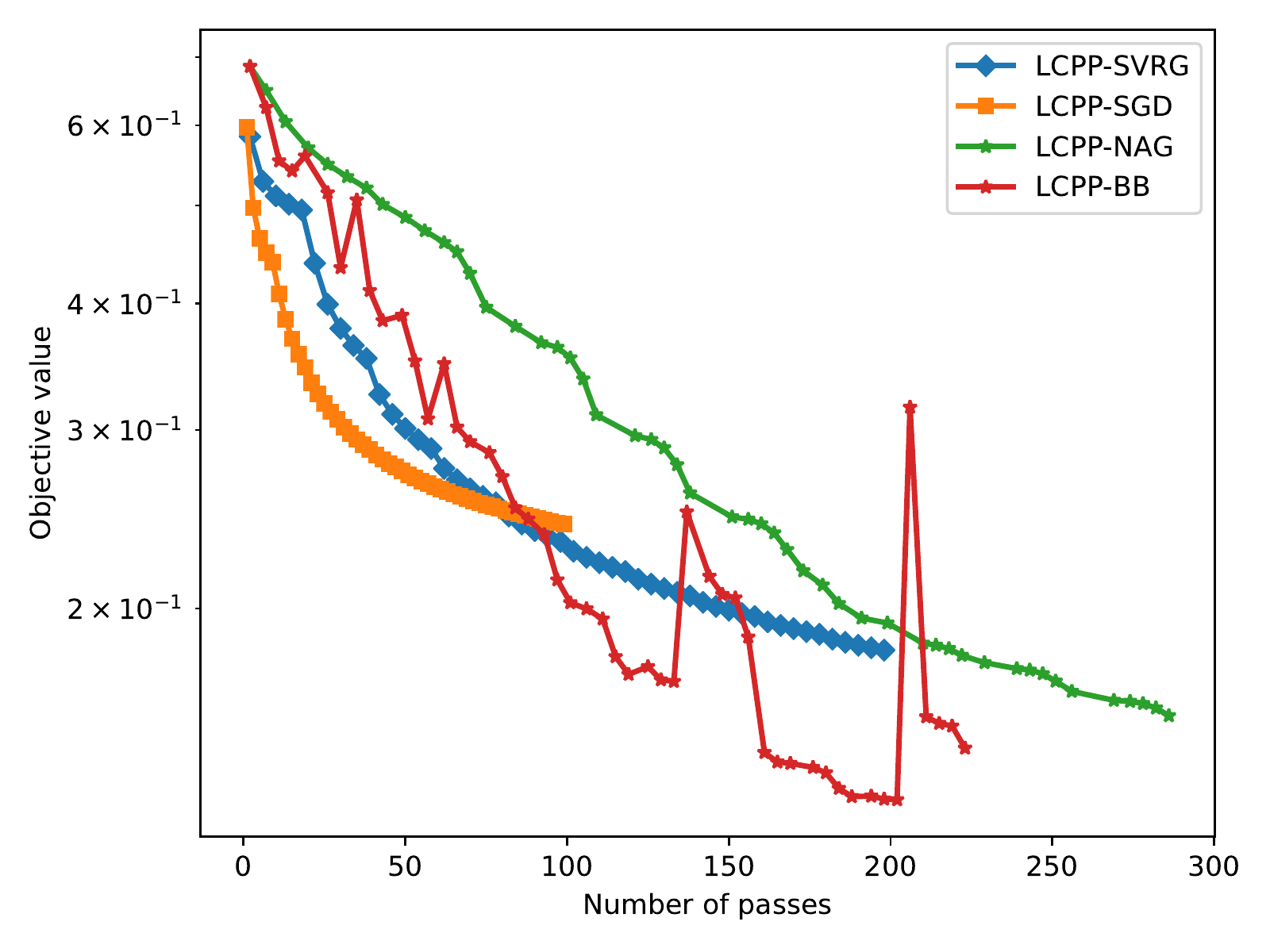}
	\includegraphics[scale=0.35]{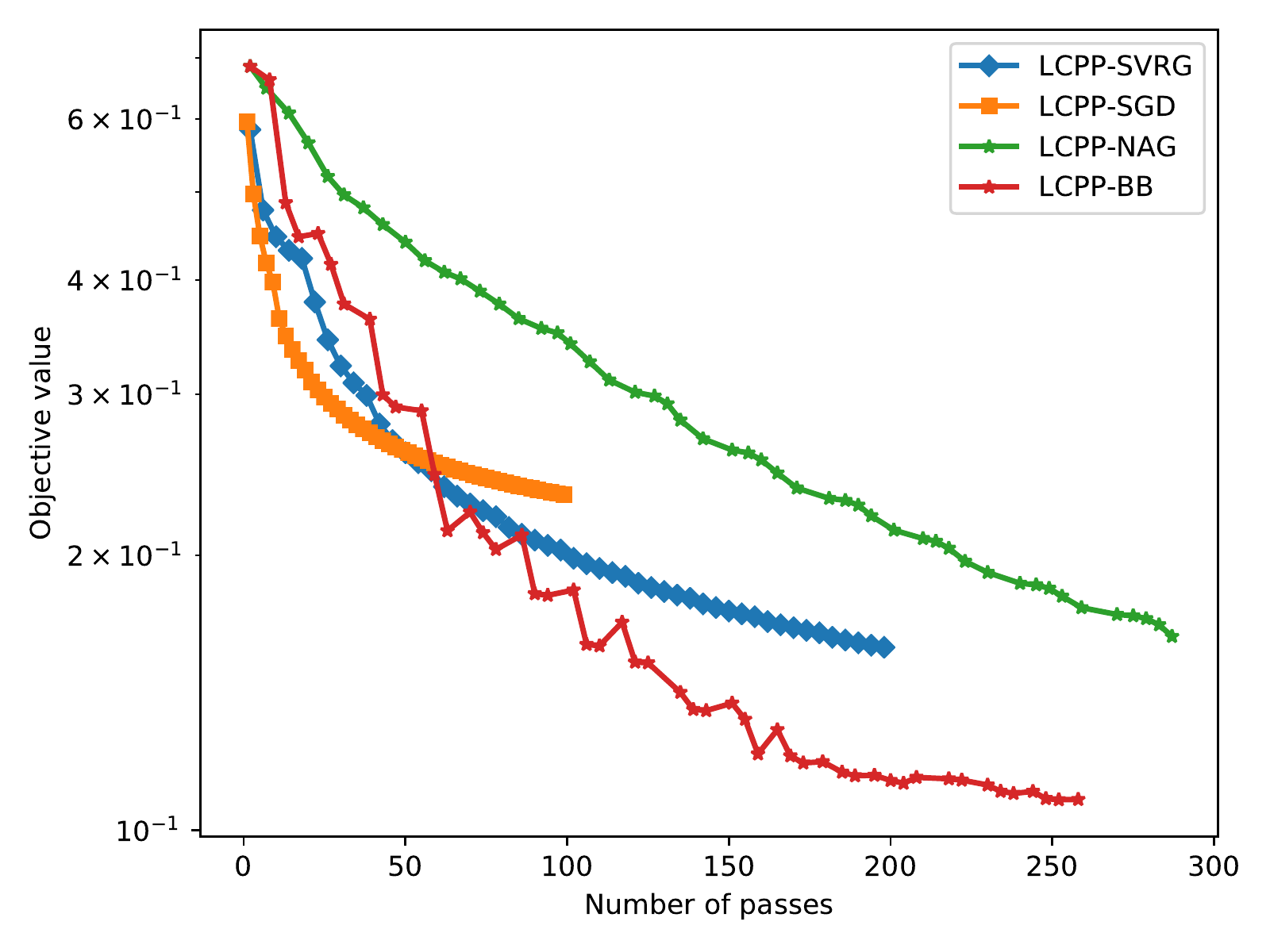}
	
	\includegraphics[scale=0.35]{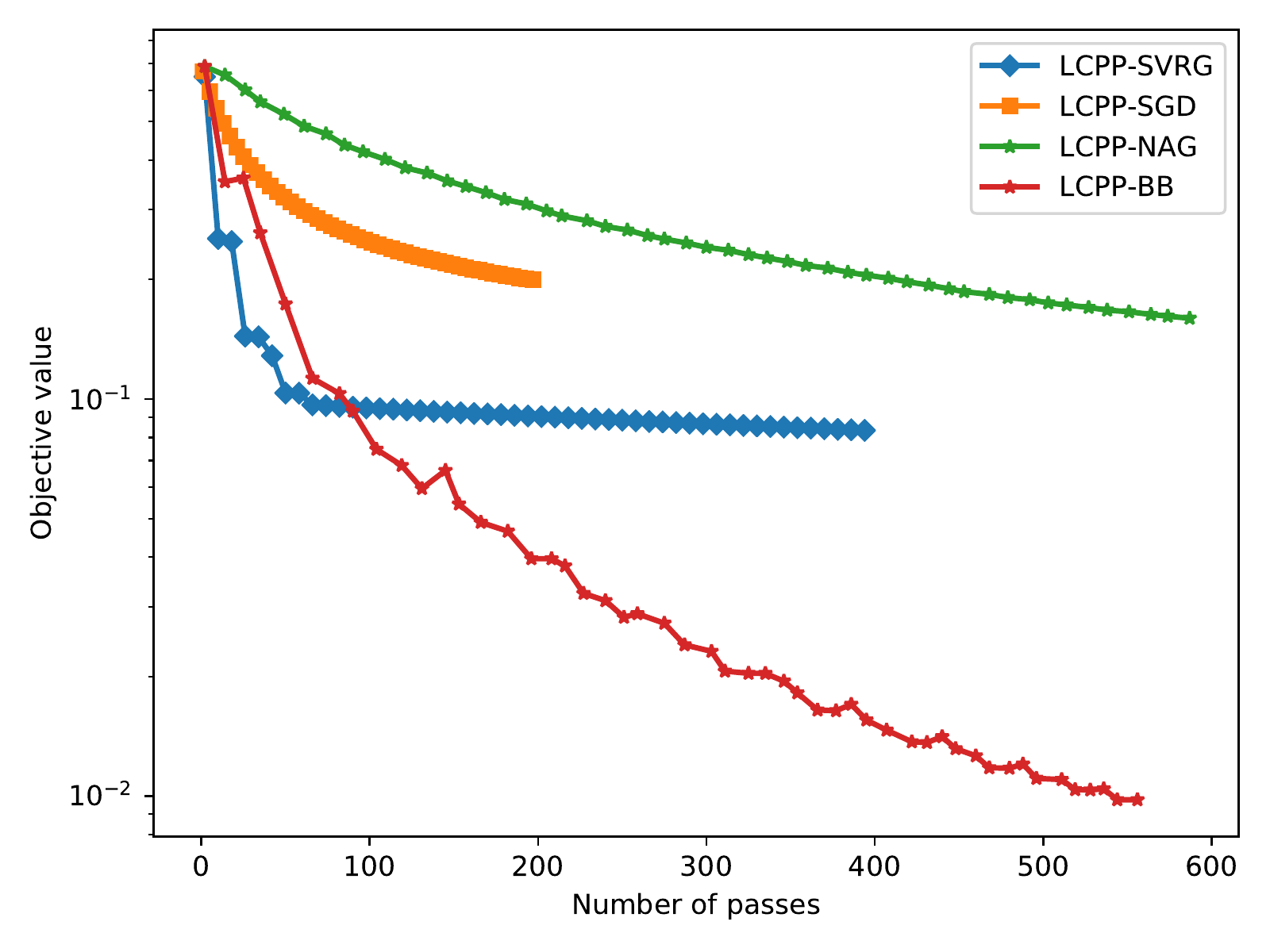}
	\includegraphics[scale=0.35]{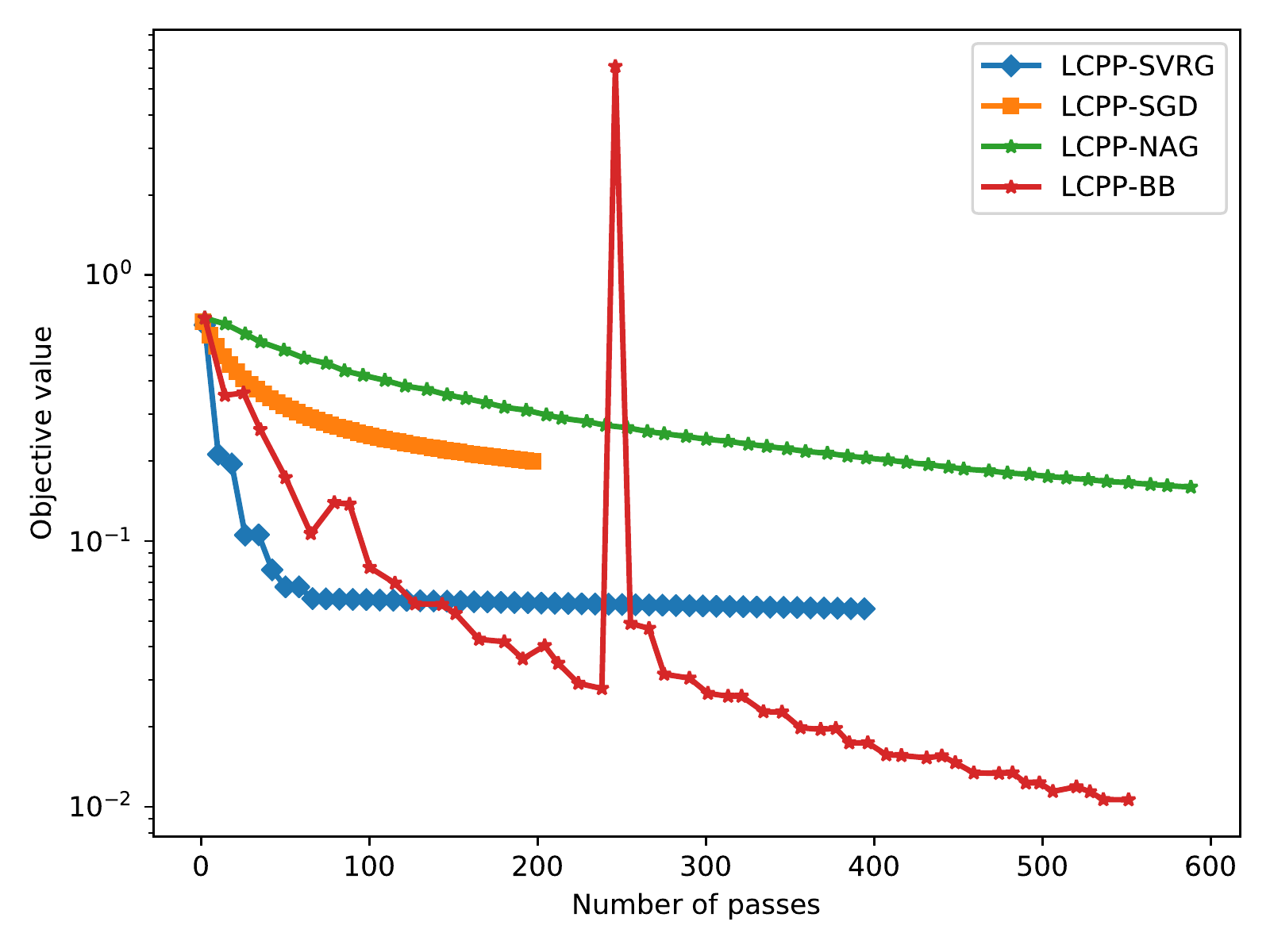}
	
	\includegraphics[scale=0.35]{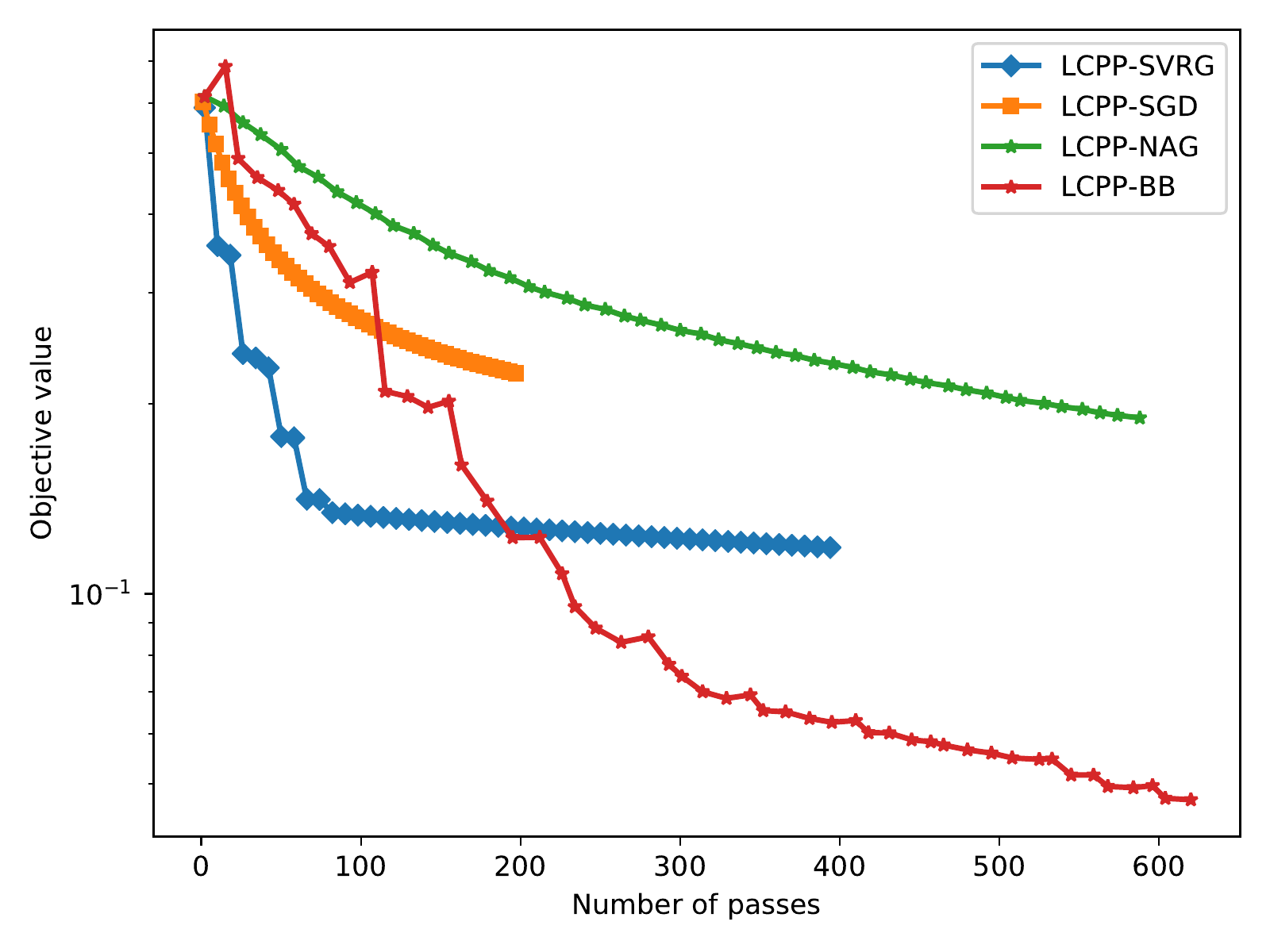}
	\includegraphics[scale=0.35]{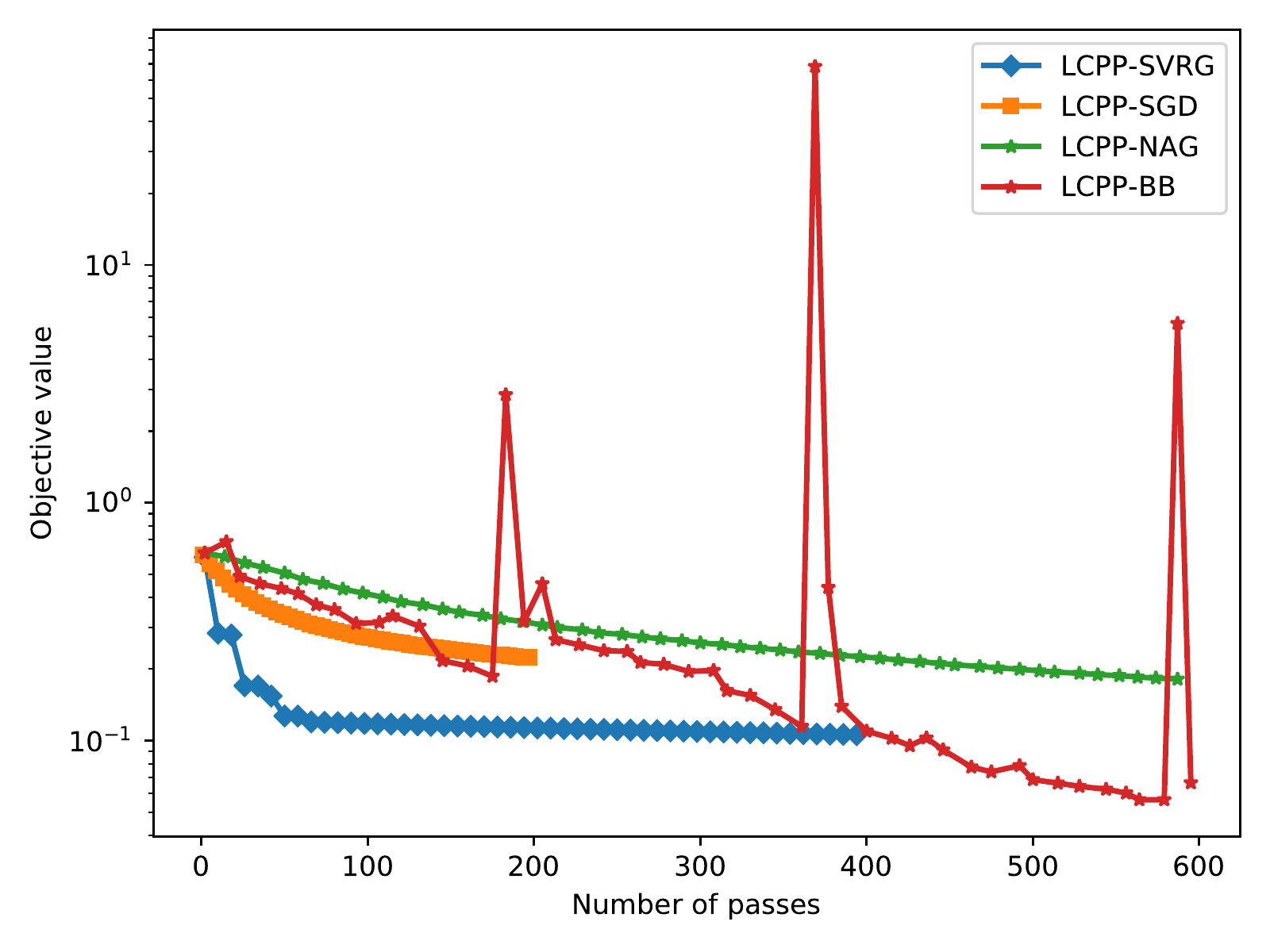}
	
	\caption{\label{fig:obj-pass-2}Objective value vs. number of effective passes over the
		dataset. Green, orange, blue and red curves represent NAG, SGD,
		SVRG and BB. We set $\eta=\alpha d $. First row: \texttt{gisette} ($\alpha=0.05,
		0.10$, left to right); second row: rcv1.binary, ($\alpha=0.10,0.20$),
		third row: real-sim ($\alpha=0.10,0.20$).}
\end{figure}

Figure~\ref{fig:obj-pass-2} shows the objective vs. number of effective
passes over the datasets. Here, each
effective pass evaluates one full gradient. 
We find that stochastic algorithms (LCPP-SGD, LCPP-SVRG) converge more rapidly than deterministic algorithms (LCPP-NAG, LCPP-BB) in the earlier stage, but they do not obtain higher accuracy in the long run. 
In all the tested datasets, we can observe that LCPP-BB outperforms the other three methods. 
Moreover, we remark that  stochastic gradient algorithms need to compute  projections more frequently than  deterministic algorithms. 
While our linesearch routine can efficiently perform projection, it is still more expensive than computing stochastic gradient, particularly, for the sparse data. Hence the overall running time of SGD algorithms is much worse than that of LCBB-BB. 
For the above reasons, we choose LCPP-BB as our default choice in the main experiment section.

\subsubsection*{Classification performance}
We conduct an additional experiment to compare the empirical performance of all the tested algorithms in sparse logistic regression.
We perform grid search based on five-fold cross-validation to find the best hyper-parameters. Then we retrain each model with the chosen hyper-parameter on the whole training dataset and report the classification performance on the testing data. Each experiment is repeated five times.
Hyper-parameters: 1) GIST: $\alpha=1$, $n\lambda\in\{10, 1, 0.1\}$ where $n$ is the size of training data, $\theta\in\{100, 10, 5, 1, 0.1, 0.01, 0.001\}$, 2) LCPP: $\lambda=2$, $\theta\in\{100, 10, 5, 1, 0.1, 0.01, 0.001\}$, $\eta=10^{-k}d$ where $k\in\{-3, -2.5,-2, -1.5, -1\}$, 3) Lasso: we set $C=C_010^s$ where $s=1+\tfrac{2}{3}k$, $k=0,1,2,...,9$, and $C_0$ is chosen by the l1\_min\_c function in Sklearn. 
 Table~\ref{tab:Testerror} summarizes the testing performance (mean and standard deviation) of each compared method. We can observe from this table that LCPP achieves the best performance on three out of the four datasets.
\begin{table}[h]
\caption{\label{tab:Testerror}Classification error ($\%$) of different methods for sparse logistic regression}	
\centering{}%
\begin{tabular}{cccc}
\toprule 
Datasets &GIST    & LCPP  &  LASSO  \tabularnewline
\midrule
\texttt{gisette}    &  $2.32 \pm 0.04$   &  $\mathbf{1.64 \pm 0.14}$  &   $1.84 \pm 0.05$  \tabularnewline
 \midrule
 \texttt{mnist}     & $2.57 \pm 0.01$    &  $\mathbf{2.52 \pm 0.02}$  &   $2.56 \pm 0.00$  \tabularnewline
 \midrule 
\texttt{rcv1.binary}&  $6.39 \pm 0.03$   &  $4.90 \pm 0.14$  &   $\mathbf{4.52 \pm 0.01}$ \tabularnewline
  \midrule
  \texttt{realsim}	& $ 3.50 \pm 0.04$   &  $\mathbf{3.03 \pm 0.00} $  &   $3.10 \pm 0.00$ \tabularnewline
  \bottomrule
\end{tabular}
\end{table}

\end{document}